\documentclass[reqno, 12pt]{amsart}
\usepackage[letterpaper,hmargin=1in,vmargin=1in]{geometry}
\usepackage{amsmath,amssymb,amsthm}
\usepackage{verbatim,epsfig,graphics}
\newtheorem{theorem}{Theorem}
\newtheorem{proposition}[theorem]{Proposition}
\newtheorem{corollary}[theorem]{Corollary}
\newtheorem{lemma}[theorem]{Lemma}
\newtheorem*{definition}{Definition}

\begin{document}
\title{The heat kernel on an asymptotically conic manifold}
\author{David A. Sher}
\date{}
\begin{abstract} In this paper, we investigate the long-time structure of the heat kernel on a Riemannian manifold $M$ which is asymptotically conic near infinity. Using geometric microlocal analysis and building on results of Guillarmou and Hassell in \cite{gh1}, we give a complete description of the asymptotic structure of the heat kernel in all spatial and temporal regimes. We apply this structure to define and investigate a renormalized zeta function and determinant of the Laplacian on $M$.
\end{abstract}
\maketitle
\section{Introduction}

We study the heat kernel on asymptotically conic manifolds. Asymptotically conic manifolds should be thought of as those complete manifolds which are approximately conic near infinity. More specifically, we have the following definition \cite{gh1}:

\begin{definition} Let $(M,g)$ be a complete Riemannian manifold without boundary of dimension $n$, and let $\bar M$ be the usual radial compactification of $M$. Let $(N,h_0	)$ be a closed Riemannian manifold of dimension $n-1$. We say that $(M,g)$ is asymptotically conic with cross-section $(N,h_0)$ if in a neighborhood of $\partial\bar M$, $\bar M$ is isometric to $[0,\delta)_x\times N_y$ with the metric
\begin{equation}\label{acmetric}g=\frac{dx^2}{x^4}+\frac{h(x)}{x^2}.\end{equation}
Here $x$ is a smooth function on $\bar M$ with $x=0$ and $dx\neq 0$ on $\partial\bar M$ (we call this a \emph{boundary defining function} for $\partial\bar M$) and a smooth family of metrics $h(x)$ on $N$ with $h(0)=h_0$. Throughout, we let $z$ be a global coordinate on $M$, writing $z=(x,y)$ in a neighborhood of the boundary of $\bar M$.
\end{definition}

In particular, Euclidean space $\mathbb R^n$ is asymptotically conic with cross-section $S^{n-1}$; we may choose $x=r^{-1}$. Any complete manifold which is exactly Euclidean or conic near infinity is, of course, also asymptotically conic. The condition (\ref{acmetric}) may be weakened by replacing $h(x)$ with any symmetric 2-tensor $h'(x,y)$ which restricts to a metric $h_0(x)$ on the boundary at $x=0$; an observation of Melrose and Wunsch \cite{mw} shows that these conditions are in fact equivalent.

Asymptotically conic manifolds are a relatively well-behaved class of manifolds, and as such the theory of the heat equation is relatively advanced. In particular, it is easy to see from (\ref{acmetric}) that all sectional curvatures of $(M,g)$ approach zero as $x$ goes to zero, and thus that $(M,g)$ has bounded sectional curvature. For complete manifolds of bounded sectional curvature, a classical theorem of Cheng-Li-Yau \cite{cly} gives the following Gaussian upper bound for the heat kernel:

\begin{theorem}\label{chengliyau}\cite{cly} For any $T>0$, there exist nonzero constants $C_1$ and $C_2$ such that the heat kernel on $M$, denoted $H^M(t,z,z')$, satisfies, for any $z,z'\in M$ and any $t\in (0,T]$,
\begin{equation}\label{gaussianub} H^M(t,z,z')\leq \frac{C_1}{t^{n/2}}e^{-\frac{|z-z'|^2}{C_2t}}.
\end{equation}
\end{theorem}

However, for many applications to spectral theory, one needs finer information about the structure of the heat kernel, often at long time. The example we have in mind is the definition of the zeta function. Recall that if $M$ is compact, the zeta function is defined for $\Re s>n/2$ by:
\begin{equation}\label{zetadefin} \zeta_M(s)=\frac{1}{\Gamma(s)}\int_0^{\infty}
(Tr H^M(t)-1)t^{s-1}\ dt.
\end{equation}
The zeta function has a well-known meromorphic continuation to all of $\mathbb C$ with a regular value at $s=0$; the key is that the trace of the heat kernel has a short-time asymptotic expansion, which, along with the long-time exponential decay, enables us to write down an explicit meromorphic continuation \cite{r}. The \emph{determinant of the Laplacian} is then given by $\exp(-\zeta'_M(0))$; the determinant plays a key role in many problems in spectral theory, including the isospectral compactness results of Osgood, Phillips, and Sarnak \cite{ops1,ops2,ops3}.

We would like to define such a zeta function and determinant when $M$ is asymptotically conic, with an eye towards applying these concepts to the spectral and scattering theory of asymptotically conic manifolds. There are several obstacles. First, the heat kernel is no longer trace class, so $Tr H^M(t)$ does not make sense. Instead, we define the \emph{renormalized heat trace} $^RTr H^M(t)$ to be the finite part at $\delta=0$ of the divergent asymptotic expansion in $\delta$ of
\begin{equation}\int_{x\geq\delta}H^M(t,z,z)\ dz.
\end{equation}
Details, including the existence of this divergent asymptotic expansion, may be found in Section 3. We then formally define the \emph{renormalized zeta function}:
\begin{equation}\label{rzetadefin} ^R\zeta_M(s)=\frac{1}{\Gamma(s)}\int_0^{\infty}\ ^RTr H^M(t)
t^{s-1}\ dt.
\end{equation}

However, to make sense of this definition and obtain a meromorphic continuation, we still need to understand the behavior of the renormalized trace - and hence of the heat kernel itself - as $t\rightarrow 0$ and $t\rightarrow\infty$. In particular, we need asymptotics in both the short and long time regimes.

The short-time behavior of the heat kernel on an asymptotically conic manifold is relatively well-understood. Short-time heat kernels may be analyzed using techniques from semiclassical analysis. In this approach, the goal is to develop a 'semiclassical functional calculus' containing the heat kernel, modeled on standard semiclassical techniques as developed, for example, in \cite{ds}. The key functional calculus for this purpose, at least in the asymptotically Euclidean setting, is the Weyl calculus of Hormander \cite{h}.

An alternate approach, and one more suited to analysis of the renormalized trace and determinant, is to use geometric microlocal analysis to first construct the heat kernel and then analyze its fine structure. The techniques of geometric microlocal analysis were first developed by Melrose and Mendoza to study elliptic PDE on manifolds with asymptotically cylindrical ends \cite{mem}. They have been extended by many other mathematicians and play a key role in the modern analysis of linear PDE on singular and non-compact spaces. In particular, in \cite{me3}, Melrose discusses some aspects of spectral and scattering theory on asymptotically conic manifolds. Albin, in \cite{alb}, uses these methods to investigate the short-time heat kernel on a variety of complete spaces, including asymptotically conic manifolds. His work can be used to obtain the fine structure that we need for the short-time heat kernel.

The long-time problem is trickier: in the asymptotically conic setting, we no longer have exponential decay of the heat kernel as $t\rightarrow\infty$. Indeed, from the structure of the Euclidean heat kernel and (\ref{gaussianub}), we expect that the leading-order behavior of $H^M(t,z,z')$ as $t\rightarrow\infty$ will be $Ct^{-n/2}$, and the leading-order behavior of the renormalized heat trace may be even worse. This lack of decay means that the zeta function may not be well-defined a priori for any $s$. We may split (\ref{rzetadefin}) into two integrals by breaking it up at $t=1$, but there is no obvious reason for the integral from $t=1$ to $\infty$ to have a meromorphic continuation to all of $\mathbb C$. In order to obtain such a meromorphic continuation, we need an asymptotic expansion for the heat kernel as $t\rightarrow\infty$. Moreover, we must understand how this expansion interacts with the heat trace renormalization.

\subsection{Main results}

We solve this problem by using the methods of geometric microlocal analysis to obtain a complete description of the asymptotic structure of the heat kernel on $M$ in all spatial and temporal regimes. The key concepts, including blow-ups and polyhomogeneous conormal functions, were originally introduced in \cite{me} and \cite{me2}, and a good introduction may be found in \cite{gri}.  In Section 2, we discuss these concepts briefly and then use them to define a new blown-up manifold with corners which we call $M^2_{w,sc}$. The space $M^2_{w,sc}$ was originally defined by Guillarmou and Hassell in \cite{gh1}, and we use their labeling of the boundary hypersurfaces; see Section 2 for the definitions. Our main theorem is the following:

\begin{theorem}\label{maincor} Let $M$ be asymptotically conic. For any $n\geq 2$, and for any fixed time $T>0$, the heat kernel on $M$ is polyhomogeneous conormal on $M^2_{w,sc}$ for $t>T$, where $w=t^{-1/2}$. The leading orders at the boundary hypersurfaces are at least 0 at sc and $n$ at each of bf$_0$, rb$_0$, lb$_0$, and zf, with infinite-order decay at lb, rb, and bf.
\end{theorem}

This theorem gives a complete description of the asymptotic structure of the heat kernel for long time; previously, only estimates such as Theorem \ref{chengliyau} were known. 
The analogous structure for the short-time heat kernel is well-understood (see Section 2 for the definition of $M^2_{sc}$):
\begin{theorem}\label{zshorttime} For $t<1$, the heat kernel
$H^{M}(t,z,z')$ is polyhomogeneous conormal on
\[[M_{sc}^{2}(z,z')\times[0,1]_{\sqrt t}; \{\sqrt t=0,z=z'\}].
\]
Moreover, there is infinite-order
decay at all faces except the scattering front face sc and the face F obtained
by the final blowup.
\end{theorem}

Theorem \ref{zshorttime} follows immediately from the work of Albin \cite{alb}; however, the precise statement above does not appear in the literature. Therefore, in the Appendix, we give a simple proof using the machinery developed in \cite{alb}. Combining Theorem \ref{maincor} with Theorem \ref{zshorttime} gives a complete geometric-microlocal description of the structure of the heat kernel. This structure is illustrated in Figure \ref{hkac}; the short-time structure is the left-hand side of the diagram, and the long-time structure is the right-hand side. We also indicate the leading order of the heat kernel at each of the boundary hypersurfaces, in terms of $\sqrt t$ at $t=0$, $w=t^{-1/2}$ at $t=\infty$, and $x$ or $x'$ respectively at all the finite-time boundaries.

\begin{figure}
\centering
\includegraphics[scale=0.80]{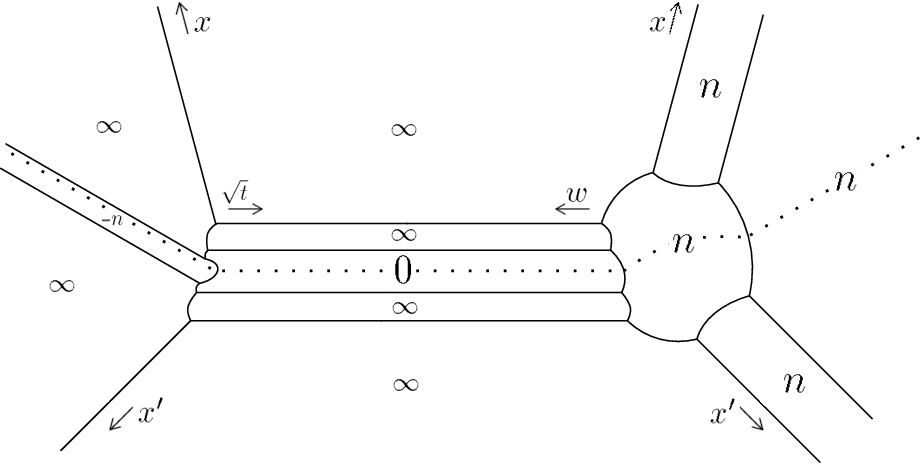}
\caption{Asymptotic structure of the heat kernel on $M$.}\label{hkac}
\end{figure}

As an example of how polyhomogeneous structure may be used to read off the behavior in all asymptotic regimes, let $a$ be a parameter, and fix $(y,x',y')$; consider the heat kernel $H^Z(a^2,a^{-1},y,x',y')$ as $a$ approaches infinity. In the compactified space in Figure \ref{hkac}, as $a$ approaches infinity, the arguments approach a point in the center of the face lb$_0$, where the leading order of the heat kernel is $n$. Since $w=a^{-1}$, and $w$ is a boundary defining function for lb$_0$, we conclude that $H^Z(a^2,a^{-1},y,x',y')$ has a polyhomogeneous asymptotic expansion in $a^{-1}$ as $a\rightarrow\infty$, with leading term $C_na^{-n}$. A similar analysis may be performed in any asymptotic regime.

As an application, Theorems \ref{maincor} and \ref{zshorttime} give us precisely the polyhomogeneous structure we need to define and investigate the renormalized zeta function on $M$:

\begin{theorem}\label{bigrenorm}  Let $M$ be asymptotically conic. The renormalized zeta function, defined formally by (\ref{rzetadefin}), is well-defined and has a meromorphic continuation to all of $\mathbb C$.
\end{theorem}

We may then define the \emph{renormalized determinant} of the Laplacian on $M$ by
\[\log(^R\det\Delta_M)=-^R\zeta'_M(0),\]
where $^R\zeta'_M(0)$ is the coefficient of $s$ in the Laurent series for $^R\zeta_M(s)$ around $s=0$.

In a companion paper \cite{s2}, we use Theorems \ref{maincor}, \ref{zshorttime}, and \ref{bigrenorm} to analyze the behavior of the determinant of the Laplacian on a family of manifolds degenerating to a manifold with conical singularities. We expect that this work will have applications to spectral theory and to index theory on singular spaces, including the study of the Cheeger-M\"uller theorem on manifolds with conical singularities. The key theorem from \cite{s2} is as follows: let $\Omega_0$ be a manifold with an exact conic singularity (with arbitrary base) and let $Z$ be a manifold conic near infinity with the same base. For each $\epsilon>0$, we define a smooth manifold $\Omega_{\epsilon}$ replacing the tip of $\Omega_0$ with an $\epsilon$-scaled copy of $Z$; as $\epsilon\rightarrow 0$, the manifolds $\Omega_{\epsilon}$ converge to $\Omega_0$ in the Gromov-Hausdorff sense. Then it is proven in \cite{s2} that
\begin{theorem}\label{detapprox}
As $\epsilon\rightarrow 0$,
\[\log\det\Delta_{\Omega_{\epsilon}}= -2\log\epsilon(^{R}\zeta_{Z}(0))
+\log\det\Delta_{\Omega_{0}}+\log^{R}\det\Delta_{Z}+o(1).\] \end{theorem}

\subsection{Outline of the proofs}
The usual geometric-microlocal approach to the fine structure of the heat kernel is a direct parametrix construction, which involves the construction of an initial approximation to the heat kernel and then the removal of the error via a Neumann series argument. This is the method adopted in \cite{alb}. However, parametrix constructions are not well-suited for analysis of the long-time heat kernel; the problem is global rather than local. In order to obtain the asymptotic structure of the heat kernel at long time, we instead take an indirect approach. Recall that the functional calculus shows that the heat kernel and the resolvent are related by
\begin{equation}\label{relation}H^{M}(t)=\frac{1}{2\pi i}\int_{\Gamma}e^{\lambda t}(\Delta_{M}+\lambda)^{-1}\ d\lambda,\end{equation}
where $\Gamma$ is a contour around the spectrum and $\Delta_M$ is the positive Laplacian.
In a series of papers \cite{gh1,gh2}, Guillarmou and Hassell have analyzed the asymptotic structure of the resolvent $(\Delta_M+\lambda)^{-1}$ at low energy, again giving a complete description in all regimes. They have shown:

\begin{theorem}\label{ghlow}\cite{gh1}
Suppose that $M$ is an asymptotically conic manifold of dimension $n\geq 3$. Then
the Schwartz kernel of $(\Delta_M+e^{i\theta}k^2)^{-1}$ is polyhomogeneous conormal on $M^2_{k,sc}$ for each $\theta\in(-\pi,\pi)$, with a conormal singularity at the spatial diagonal and all coefficients smoothly depending on $\theta$. It decays to infinite order at the faces lb, rb, and bf, with leading orders at sc, bf$_0$, rb$_0$, lb$_0$ and zf given by $0$, $n-2$,
$n-2$, $n-2$, and $0$ respectively.\end{theorem}

Note that Guillarmou and Hassell require $n\geq 3$. In Section 4,
we adapt the methods of \cite{gh1} to extend Theorem \ref{ghlow} to the two-dimensional case:
\begin{theorem}\label{dslow} Theorem \ref{ghlow} also holds when $n=2$; all the leading orders are the same, except that we have logarithmic growth instead of order 0 at zf. \end{theorem}

In Section 2, we use geometric microlocal analysis, in particular Melrose's pushforward theorem, to prove Theorem \ref{maincor}. The key is to push the structure of Theorems \ref{ghlow} and \ref{dslow} through the contour integral (\ref{relation}). To state the main technical theorem, we must first compactify $M^2_{k,sc}$ to $\bar M^2_{k,sc}$ by introducing a new boundary face at $k=\infty$, with boundary defining function $k^{-1}$. In a neighborhood of the new face, which we call tf, $\bar M^2_{k,sc}$ is $M^2_{sc}\times[0,1)_{k^{-1}}$. The main technical theorem is:

\begin{theorem}\label{mainthm} Let $M$ be an asymptotically conic manifold, and let $E$ be a vector bundle over $M$. Let $A:C^\infty(E)\rightarrow C^{-\infty}(E)$ be a pseudodifferential operator with the following properties:

a) $\sigma(A)\subset[0,\infty]$;

b) (Low energy resolvent behavior) For $k$ bounded above, the Schwartz kernel of the resolvent $(A+e^{i\theta}k^2)^{-1}$ is polyhomogeneous conormal on $M^2_{k,sc}$ for each $\theta\in(-\pi,\pi)$, with a conormal singularity at the spatial diagonal and all coefficients smoothly depending on $\theta$. Moreover, it decays to infinite order at the faces lb, rb, and bf, with index sets at sc, bf$_0$, rb$_0$, lb$_0$, and zf given by $R_{sc}$, $R_{bf_0}$,
$R_{rb_0}$, $R_{lb_0}$, and $R_{zf}$ respectively.

c) (High energy resolvent behavior) For each $\theta\in(-\pi,\pi)$ and for $k$ bounded below, the Schwarz kernel of $(A+e^{i\theta}k^2)^{-1}$ is phg conormal on $\bar M^2_{k,sc}$, with infinite-order decay at lb, rb, and bf, index set $R_{sc}$ at sc, and index set $R_{tf}$ at tf.

Then for $t$ greater than any fixed $T>0$, the kernel of $e^{-tA}$ is polyhomogeneous conormal on $M^2_{w,sc}$, where $w=t^{-1/2}$. It decays to infinite order at lb, rb, and bf, and has
index sets at sc, bf$_0$, rb$_0$, lb$_0$, and zf which are subsets of $R_{sc}$, $R_{bf_0}+2$, $R_{rb_0}+2$, $R_{lb_0}+2$, and $R_{zf}+2$ respectively.
\end{theorem}

Once we have proven Theorem \ref{mainthm}, Theorem \ref{maincor} is an almost immediate consequence, though there is a slight twist involving the leading orders.

In section 3, we use Theorem \ref{maincor} and some additional geometric microlocal techniques to analyze the renormalized heat trace and prove Theorem \ref{bigrenorm}. We also analyze the renormalized zeta function and determinant in the special case where $M$ is exactly conic (or Euclidean) outside a compact set. In section 4, we extend Guillarmou and Hassell's work in \cite{gh1} to prove Theorem \ref{dslow}. Finally, in the Appendix, we use the framework in \cite{alb} to prove Theorem \ref{zshorttime}.

\subsection{Acknowledgements}

This work comprises the first part of my Stanford Ph.D. thesis \cite{s}. I owe a great deal to my advisor, Rafe Mazzeo, who introduced me to this problem and shared tremendous advice and support. It would be impossible to thank everyone who contributed to this project, but I would like to particularly thank Pierre Albin, Colin Guillarmou, Andrew Hassell, Andras Vasy, and the anonymous referee for providing interesting and helpful ideas. I am grateful to Gilles Carron and the Universit\'e de Nantes for their generous hospitality during the fall of 2010, during which a significant portion of this work was completed. Finally, I would like to thank the ARCS foundation for financial support in the 2011-2012 academic year.

\section{From resolvent to heat kernel}

The goal of this section is to prove Theorems \ref{mainthm} and \ref{maincor}. 

\subsection{Preliminaries}

We first give a brief summary of the key relevant concepts in geometric microlocal analysis; again, a self-contained introduction may be found in \cite{gri}. 
A \emph{manifold with corners} of dimension $n$ is a topological space which is locally modeled on 
$\mathbb R^k_+\times\mathbb R^{n-k}$ for some $k$; a simple example is the $n$-dimensional unit cube. \emph{Blow-up} is a way of creating new manifolds with corners from old ones, and is used to resolve certain geometric singularities. The idea is to formally introduce polar coordinates around a submanifold of a manifold with corners, in order to distinguish between directions of approach to that submanifold. For example, consider the origin as a submanifold of $\mathbb R_+^2$. To blow up the origin, we introduce polar coordinates $(r,\theta)$, which corresponds to replacing the point $(0,0)$ with a quarter-circle, which corresponds to the inward-pointing spherical normal bundle of $(0,0)\subset\mathbb R_+^2$. See \cite{me}, \cite{me2}, \cite{gri}, and/or the appendix of \cite{s} for a more detailed explanation and more general examples of blow-ups.

By Taylor's theorem, smooth functions on a manifold with corners are precisely those functions which have Taylor expansions at each boundary hypersurface and joint Taylor expansions at every corner. \emph{Polyhomogeneous conormal distributions}, which we abbreviate as phg or phg conormal, are a generalization of smooth functions. In particular, if we let $x$ be a boundary defining function, we allow terms of the form $x^s(\log x)^p$ for any $x\in\mathbb C$ and any $p\in\mathbb N_0$ to appear in the asymptotic expansions at the boundary and in the joint expansions at the corners. The \emph{index set} of a phg conormal distribution $u$ at a particular boundary hypersurface $H$ is simply the set of $(s,p)$ which appear in the asymptotic expansion of $u$ at $H$.

Polyhomogeneous conormal functions are well-behaved under addition and multiplication, but also under more complicated operations, namely pull-back and push-forward. To discuss these, we first need to discuss properties of a map $f:W\rightarrow Z$ between manifolds with corners. Roughly, we say that $f$ is a \emph{b-map} if it is smooth up to the boundary and product-type near the boundary in terms of the local coordinate models (see \cite{gri} for a precise definition). If additionally $f$ does not map any boundary hypersurface of $W$ into a corner of $Z$, and $f$ is also a fibration over the interior of every boundary hypersurface, then we call $f$ a \emph{b-fibration}. We now have the following two theorems, due to Melrose \cite{me4}, which will be critical in the analysis to follow:

\begin{proposition}[Melrose's Pull-Back and Push-Forward Theorems] Let $f:W\rightarrow Z$ be a smooth map of manifolds with corners. 

a) If $f$ is a b-map and $u$ is phg conormal on $Z$, then $f^*u$ is phg conormal on $W$. Moreover, the index sets of $f^*u$ may be computed explicitly from those of $u$ and the geometry of the map $f$.

b) If $f$ is also a b-fibration, $v$ is phg conormal on $W$, and $f_*v$ is well-defined (the pushforward is integration along the fibers, which may not converge), then $f_*v$ is phg conormal on $Z$, and again the index sets may be computed explicitly.
\end{proposition}

Finally, we need to consider distributions which have pseudodifferential-type \emph{conormal singularities} at submanifolds in the interior of a manifold with corners.
\begin{definition}\cite{gri} Let $y=(x_{1},\ldots,x_{k})$ and 
$z=(x_{k+1},\ldots,x_{n})$, and let $N$ be the set $\{z=0\}$ in $\mathbb R^n_{y,z}$. A distribution $u$ on $\mathbb R^{n}$ has a
conormal singularity at $N$ of order $m$ if it can be written
\[u(y,z)=\int_{\mathbb R^{n-k}}e^{iz\cdot\xi}a(y,\xi)\ d\xi,\]
where $a$ is a classical symbol; that is, $a$ has asymptotics as $|\xi|
\rightarrow\infty$
\[a(y,\xi)\sim\sum_{j=0}^{\infty}a_{m-j}(y,\frac{\xi}{|\xi|})|\xi|^{m-j},\]
with each coefficient $a_{m-j}$ smooth in $y$ and $\frac{\xi}{|\xi|}$.
\end{definition}
This definition may be extended, by using the local coordinate models, to define distributions with a conormal singularity at any \emph{p-submanifold} of a manifold with corners; a p-submanifold is a subset which, in each local coordinate chart, may be identified with a coordinate submanifold. Variants of the pullback and pushforward theorems also hold for polyhomogeneous conormal distributions with interior conormal singularities \cite{me2,emm}.

\subsection{The space $M^2_{k,sc}$}
We now introduce the space $M^2_{k,sc}$, which first appears in an unpublished note of Melrose and Sa Barreto \cite{msb}, and later in \cite{gh1, gh2, ghs}. To construct $M^2_{k,sc}$, we begin with the space
$M^2_k=[0,\infty)_k\times\bar M\times \bar M$; coordinates on this space near $[0,\infty)_k\times\partial\bar M\times\partial\bar M$ are $(k,x,y,x',y')$.
There are three boundary hypersurfaces: $\{k=0\}$, which we call zf,
$\{x=0\}$, which we call lb, and $\{x'=0\}$, which we call rb.

\begin{figure}
\centering
\includegraphics[scale=0.45]{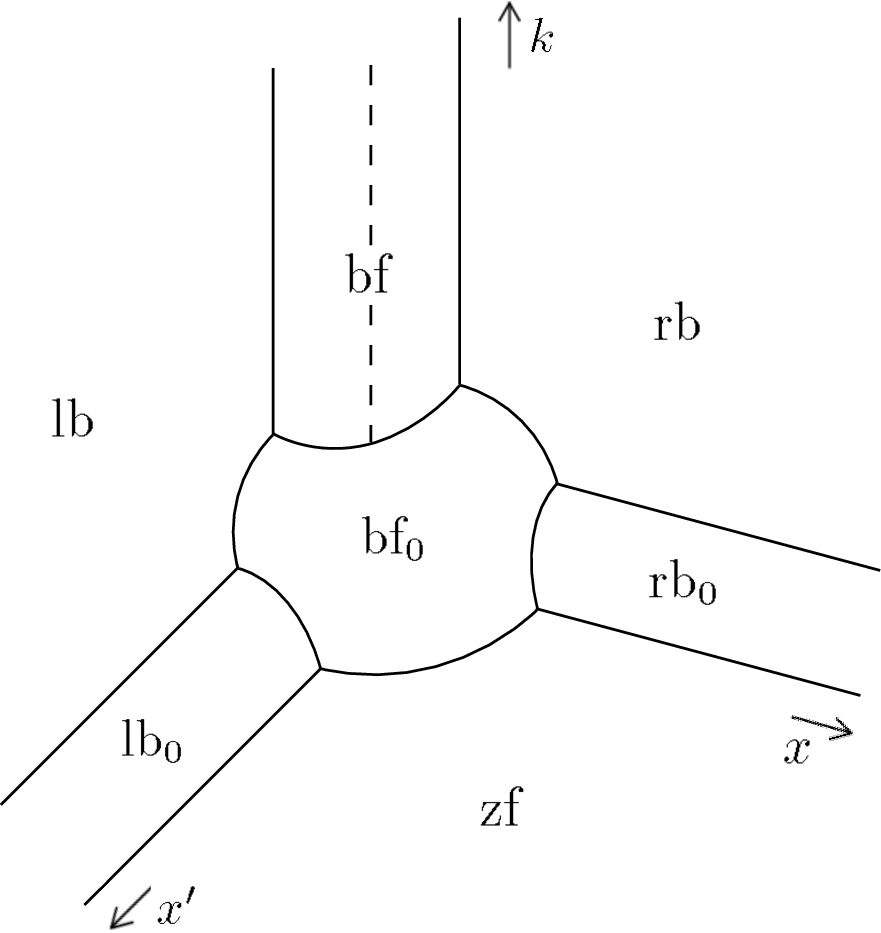}
\caption{The space $M^2_{k,b}$.}\label{firstblowupgh}
\end{figure}

First we blow up the corner $\{x=0,x'=0,k=0\}$, which corresponds to the introduction of polar coordinates near that corner; we call the front face of this
blowup bf$_{0}$. We then blow up three codimension-2 submanifolds: we blow up
$\{x=0,x'=0\}$ and call the new face bf, we blow up $\{x=0,k=0\}$ and call
the resulting face lb$_{0}$, and we blow up $\{x'=0,k=0\}$ and call the resulting
face rb$_{0}$. The resulting manifold with corners, which we call $M^2_{k,b}$ (as in \cite{ghs}), is shown in Figure \ref{firstblowupgh}, with $y$ and $y'$ suppressed. Using the definition of a "b-stretched product" from the work of Melrose-Singer \cite{ms}, we can identify this manifold near $x=0$, $x'=0$ with $X_{b}^{3}(x,x',k)\times N_{y}\times N_{y'}$. We use these b-stretched products $X_b^n$ from \cite{ms} throughout the arguments.

Finally, consider the intersection of the closure of the interior spatial
diagonal with the face bf. In coordinates near the boundary,
this is $\{(x/x')=1,y=y',(x/k)=0\}$, and it is marked with a dotted line in 
Figure \ref{firstblowupgh}. We blow this up to create a new boundary hypersurface, which we call
sc (for ``scattering''). The resulting space is $M^{2}_{k,sc}$,
and it has eight boundary hypersurfaces, illustrated in Figure \ref{ghspace}.
The spatial diagonal $D_{k,sc}$ is defined to be 
the closure in $M^{2}_{k,sc}$ of the
interior spatial diagonal; its intersection with the boundary
is marked with a dotted line in Figure \ref{ghspace}.

\begin{figure}
\centering
\includegraphics[scale=0.45]{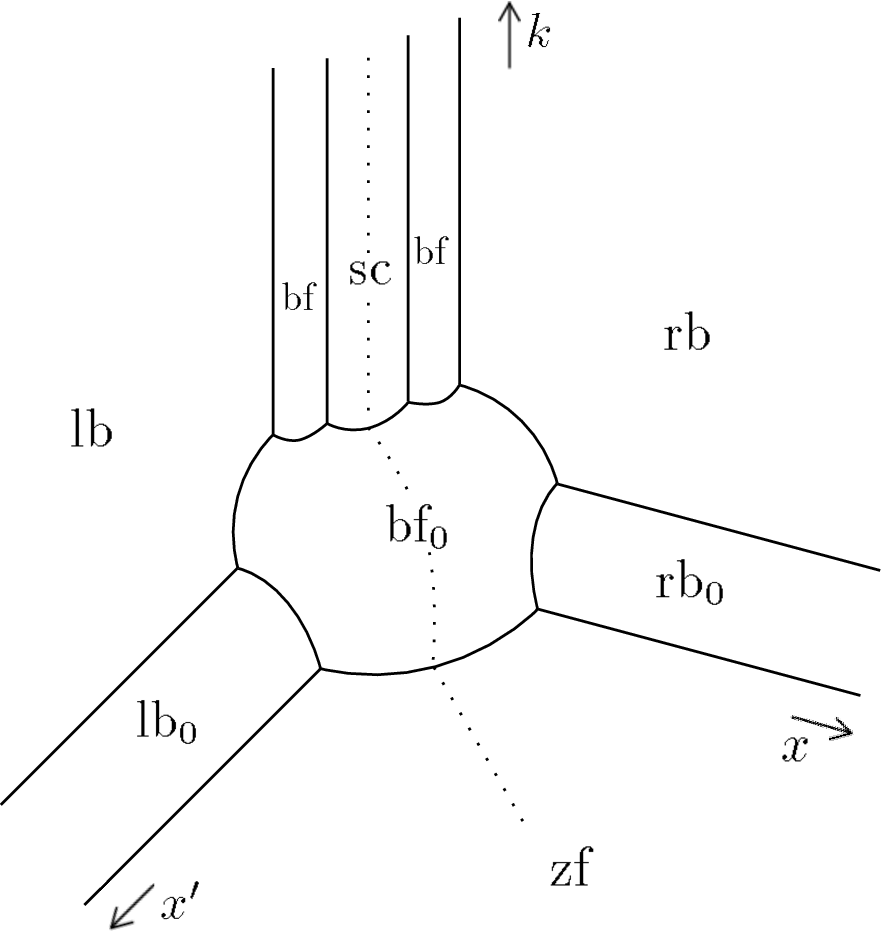}
\caption{The space $M^{2}_{k,sc}$.}\label{ghspace}
\end{figure}

We now describe some useful coordinate systems on $M^{2}_{k,sc}$. Near
the intersection of zf, rb$_{0}$, and bf$_{0}$, we use the coordinates
\[(x,\sigma=\frac{x'}{x},y,y',\kappa'=\frac{k}{x}).\]
In these coordinates, $x$ is a boundary defining function (bdf) for bf$_{0}$, $\sigma$ is a bdf for rb$_{0}$,
and $\kappa'$ is a bdf for zf. Similarly, near the intersection of zf, lb$_{0}$, and bf$_{0}$,
we use the coordinates
\[(x',\sigma'=\frac{x}{x'},y,y',\kappa=\frac{k}{x}).\]

Coordinates near sc are slightly more complicated; before the final blowup,
good coordinates are $(x',\sigma',y,y',(x/k))$. After the blowup, we
use the coordinates
\[X=k(\frac{1}{x'}-\frac{1}{x}),Y=k\frac{y-y'}{x},\lambda=\frac{x}{k},y,k.\]
These are valid in a neighborhood of the intersection of $D_{k,sc}$
with sc and bf$_{0}$; however, they are not good coordinates as we approach bf.

In addition to the b-stretched products of \cite{ms} and the space $M^{2}_{k,sc}$, we also define the \emph{scattering double space} $M^2_{sc}(z,z')$, originally described in \cite{me3} (see also \cite{ms}). It is a blown-up version of $\bar M\times \bar M$; the first blow-up is of $\{x=x'=0\}$, and the second blow-up is of the boundary fiber diagonal $\{x'=0, x/x'=1, y=y'\}$. Notice that each cross-section of $M^2_{k,sc}$ corresponding to a fixed $k>0$ is a copy of $M^2_{sc}$.
\subsection{Proof of Theorem \ref{mainthm}}

Let $A$ be an operator satisfying hypotheses a)-c) of Theorem \ref{mainthm}, and let $R(\lambda,z,z')$ be the Schwartz kernel of $(A+\lambda)^{-1}$. The spectrum of $A$ is $[0,\infty]$, so $R(\lambda,z,z')$ is holomorphic outside the non-positive real axis. Fix $\varphi\in(\pi/2,\pi)$.
For any $a>0$, let $\Gamma_{a}$ be the path in $\mathbb C$
consisting of two half-rays along $\theta=-\varphi$ and
$\theta=\varphi$, connected by the portion of the circle of radius $a$ 
from $\theta=-\varphi$ to $\theta=\varphi$, and traversed counterclockwise.
Moreover, let $\Gamma_{a,1}$ be the portion of $\Gamma_{a}$ along the circle
of radius $a$, and let $\Gamma_{a,2}$ be the remainder; that is, the two
half-rays. These contours are illustrated in Figure \ref{contour}.

\begin{figure}
\centering
\includegraphics[scale=0.70]{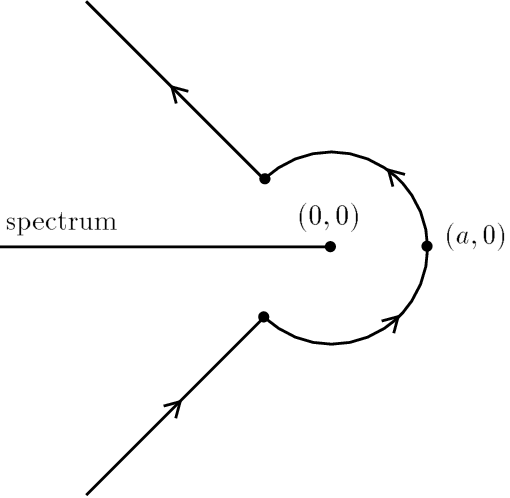}
\caption{Contour of integration $\Gamma_{a}$.}\label{contour}
\end{figure}

Let $F(w,z,z')$ be the heat kernel at time $t=w^{-2}$. Then by the functional calculus, we have
\begin{equation}\label{bigint}F(w,z,z')=\frac{1}{2\pi i}\int_{\Gamma}e^{\lambda/w^{2}}R(\lambda,z,z')\ d\lambda.\end{equation}
We let $a=w^{2}$ and $\Gamma=\Gamma_{w^{2}}$, and then consider the integral 
(\ref{bigint})
over $\Gamma_{w^{2},1}$ and $\Gamma_{w^{2},2}$ separately.

On $\Gamma_{w^{2},1}$, $\lambda=w^{2}e^{i\theta}$, so $d\lambda=w^{2}d\theta$,
and we have \begin{equation}\label{gammaone}
\frac{w^{2}}{2\pi i}\int_{-\varphi}^{\varphi}e^{e^{i\theta}}R(\theta,w,z,z')\ d\theta.
\end{equation}
By condition b), for each $\theta$, the integrand in (\ref{gammaone})
is phg conormal on $M^{2}_{w,sc}$ with a conormal singularity at $\Delta_{w,sc}$, 
and the dependence of all coefficients on $\theta\in [-\varphi,\varphi]$ is smooth. Therefore, the integral (\ref{gammaone}) is phg conormal on $M^{2}_{w,sc}$, with a possible
conormal singularity at $D_{w,sc}$. The index sets of (\ref{gammaone}) on $M^2_{w,sc}$
are those of $w^{2}$ plus those of $R(w,z,z')$. The function 
$w^{2}$ is smooth and has order
2 as a function at zf, bf$_{0}$, rb$_{0}$, and lb$_{0}$ and order 0 everywhere
else, so we add 2 to the index sets of the resolvent at those faces. This procedure gives precisely the index sets claimed in Theorem \ref{mainthm}.

It remains to consider the integral over $\Gamma_{w^{2},2}$. 
$\Gamma_{w^{2},2}$ consists
of two half-rays; we consider only the half-ray corresponding to $\theta=\varphi$,
as the other is analogous. Since $\theta$ is fixed, we suppress $\theta$
in the notation; the integral runs from $r=w^{2}$ to $r=\infty$.
After changing variables from $r$ to $s=\sqrt r$, and dropping the overall factor of $(2\pi i)^{-1}$ (which does not affect polyhomogeneity), the $\Gamma_{w^{2},2}$ portion
of (\ref{bigint}) becomes: \begin{equation}\label{gammatwo}
\int_{w}^{\infty}2se^{(\cos\varphi)s^{2}/w^{2}}e^{i(\sin\varphi)s^{2}/w^{2}}
R(s^2,z,z')\ ds.\end{equation}

First consider the behavior of the integrand as $s\rightarrow\infty$. Since $\cos\varphi<0$ and $w$ is bounded above, the term $\exp((\cos\varphi)s^2/w^2)$, and hence the entire integrand of (\ref{gammatwo}), will decay to infinite order at $s=\infty$. There will still be a conormal singularity at the spatial diagonal, but the coefficients decay to infinite order at $s=\infty$.

In order to analyze (\ref{gammatwo}), we break $R(s^2,z,z')$ into pieces. First we
separate out the conormal singularity in a neighborhood of $D_{s,sc}$ and analyze
(\ref{gammatwo}) using explicit local coordinates. Then we deal
with the remainder, which is smooth on $M^{2}_{s,sc}$, by breaking it into
two pieces and applying the pushforward theorem.

\subsection{The conormal singularity}

Using a partition of unity, we let $R=R_{s}+R_{c}$, where $R_{c}$ is supported
in a neighborhood of $D_{s,sc}$ and $R_{s}$ is supported away from $D_{s,sc}$, as in Figure
5.4. In a neighborhood of sc$\cap$bf$_{0}$, we use the coordinates
\[(X=s(\frac{1}{x}-\frac{1}{x'}), Y=s\frac{y-y'}{x},y,\mu=\frac{x}{s},s).\]
In these coordinates, $R$ has a conormal singularity at $\{X=Y=0\}$.
On the other hand, in a neighborhood of bf$_{0}\cap$zf$\cap D_{s,sc}$, 
we use a slight modification of the coordinates in the previous subsection: \[(\hat X=(1-\frac{x}{x'}),\hat Y=y-y',y,\frac{s}{x},x).\]
Since we use different coordinates in different regimes, write
$R_{c}=R_{1}+R_{2}+R_3$ by using a smooth partition of unity near $x/s=1$;
$R_{1}$ is supported near the boundary but away from sc (say $s/x<2$), $R_{2}$ near the boundary but away from zf (say $x/s<2$), and $R_3$ in the interior. The decomposition at the boundary is illustrated in Figure \ref{rdecomp}.

\begin{figure}
\centering
\includegraphics[scale=0.70]{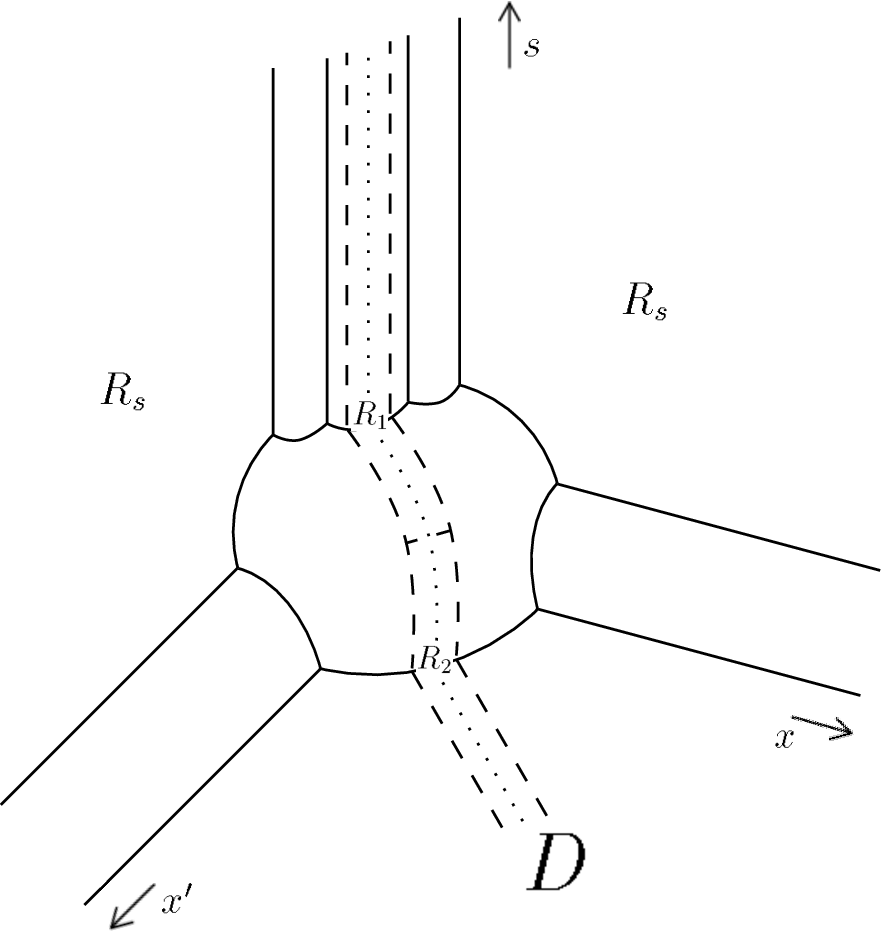}
\caption{Decomposition of $R$.}\label{rdecomp}
\end{figure}

First look at $R_{1}$. Using the explicit symbolic form of a conormal singularity, we may write
\begin{equation}\label{finalthing}
R_{1}\sim\int_{\mathbb R^{n}}e^{i(\hat X,\hat Y)\cdot(\xi_{1},\xi_{2})}
\sum_{j=0}^{\infty}a_{j}(\frac{s}{x},x,y,\frac{\xi}{|\xi|})|\xi|^{2-j}\ d\xi.
\end{equation}
This is an asymptotic sum, modulo smooth functions on $M^{2}_{s,sc}$;
we pick a particular representative which is supported
in a small neighborhood of $D_{s,sc}$, and absorb the remainder into $R_{s}$. The
coefficients $a_{j}$ are phg conormal in $x$ and $s/x$ with
index sets independent of $j$; they are also smooth in $y$ and $\frac{\xi}{|\xi|}$.
We plug (\ref{finalthing}) into (\ref{gammatwo}) and then interchange the convergent 
$s$-integral with the asymptotic sum and the oscillatory integral 
over $\mathbb R^{n}$. The result is
\begin{equation}\label{conormintone}
\int_{\mathbb R^{n}}e^{i(\hat X,\hat Y)\cdot(\xi_{1},\xi_{2})}
\sum_{j=0}^{\infty}(\int_{w}^{\infty}2se^{-(s/w)^{2}e^{i\varphi}}
a_{j}(\frac{s}{x},x,y,\frac{\xi}{|\xi|})\ ds)|\xi|^{2-j}\ d\xi.\end{equation}

By Melrose's pullback theorem \cite{me2},
the pullback of each $a_j$ to $X_{b}^{3}(s,x,w)\times N_{y}\times
S^{n-1}_{\xi/|\xi|}$ via projection is also phg conormal with index sets
independent of $j$. As a result, the integrand in
\begin{equation}\label{conormcoeffone}
\int_{w}^{\infty}2se^{-(s/w)^{2}e^{i\varphi}}
a_{j}(\frac{s}{x},x,y,\frac{\xi}{|\xi|})\ ds\end{equation}
is phg conormal on $X_{b}^{3}(s,w,x)\times N_{y}\times S^{n-1}_{\xi/|\xi|}$, 
with a cutoff singularity at $s/w=1$. Moreover, it has infinite-order decay at $s=\infty$, independent of $w<1$ and $x<1$, and hence integration in $s$ is well-defined. Integration in $s$ is a b-fibration from $X_{b}^{3}(s,w,x)$
to $X_{b}^{2}(w,x)$, by Proposition 4.4 of \cite{ms}, and is hence also a b-fibration
when we take the direct product with $N_{y}\times S^{n-1}_{\xi/|\xi|}$. Moreover,
integration in $s$ is transverse to the cutoff singularity at $s/w=1$.
Therefore, by the pushforward theorem with conormal singularities (from the Appendix of \cite{emm}),
(\ref{conormcoeffone}) is phg conormal on $X_{b}^{2}(w,x)\times N_{y}\times
S^{n-1}_{\xi/|\xi|}$, with index sets independent of $j$.

Since $w<s<2x$ on the support of $R_{1}$,
(\ref{conormcoeffone}) has a phg conormal expansion in $(w/x,x)$.
Therefore, the integral (\ref{conormintone}), in the coordinates $(\hat X,\hat Y,y,w/x,x)$, corresponding to the
$R_{1}$ piece of (\ref{gammatwo}), is phg conormal in $(w/x,x)$ with index sets independent of $j$, and
smoothly dependent on $y$, with an
interior conormal singularity at $\hat X=\hat Y=0$. Thus
(\ref{conormintone}) is phg conormal on $M^{2}_{w,sc}$ with a conormal
singularity at the diagonal.

\medskip

We now consider $R_{2}$; the analysis is similar. Write
\begin{equation}\label{conormformtwo}
R_{2}\sim\int_{\mathbb R^{n}}e^{i(X,Y)\cdot(\xi_{1},\xi_{2})}
\sum_{j=0}^{\infty}b_{j}(\frac{x}{s},s,y,\frac{\xi}{|\xi|})|\xi|^{2-j}\ d\xi,
\end{equation}
where the $b_{j}$ are phg conormal in $\frac{x}{s}$ and $s$ with index sets
independent of $j$, and also smooth in $y$ and $\frac{\xi}{|\xi|}$.

It is helpful to consider the regimes $w>x/2$ and $w<2x$ separately. 
First assume that $w>x/2$, and let $\bar X=w(\frac{1}{x}-\frac{1}{x'})$,
$\bar Y=w\frac{y-y'}{x}$, and $\bar\lambda=\frac{x}{w}$. We expect a
conormal singularity at $\bar X=\bar Y=0$ in this regime. Noting that 
$(X,Y)=(s/w)(\bar X,\bar Y)$, we change variables in (\ref{conormformtwo})
and let $\zeta=(s/w)\xi$. The result is:
\begin{equation}\label{conormmodformtwo}
R_{2}\sim\int_{\mathbb R^{n}}e^{i(\bar X,\bar Y)\cdot(\zeta_{1},\zeta_{2})}
\sum_{j=0}^{\infty}b_{j}(\frac{x}{s},s,y,\frac{\zeta}{|\zeta|})(s/w)^{j-2-n}
|\zeta|^{2-j}\ d\zeta.\end{equation}
As before, plug (\ref{conormmodformtwo}) into (\ref{gammatwo}) and
interchange the sums and convergent integrals: the part of (\ref{gammatwo})
coming from $R_{2}$ is
\begin{equation}\label{conorminttwo}
\int_{\mathbb R^{n}}e^{i(\bar X,\bar Y)\cdot(\zeta_{1},\zeta_{2})}
\sum_{j=0}^{\infty}(\int_{w}^{\infty}2se^{-(s/w)^{2}e^{i\varphi}}
b_{j}(\frac{x}{s},s,y,\frac{\zeta}{|\zeta|})(s/w)^{j-2-n}\ ds)|\zeta|^{2-j} d\zeta.
\end{equation}
 
Consider the coefficients
\begin{equation}\label{conormcoefftwo}
\int_{w}^{\infty}2se^{-(s/w)^{2}e^{i\varphi}}
b_{j}(\frac{x}{s},s,y,\frac{\zeta}{|\zeta|})(s/w)^{j-2-n}\ ds.
\end{equation}
If we can show that the coefficients (\ref{conormcoefftwo}) are 
phg conormal in $(\frac{x}{w},w)$ with
respect to some index sets independent of $j$, with smooth
dependence on $y$ and $\frac{\zeta}{|\zeta|}$, then (\ref{conorminttwo})
is phg conormal on $M^{2}_{w,sc}$ with a conormal singularity
at the diagonal when $w>x/2$.
To show this phg conormality, note that again the integrands in
(\ref{conormcoefftwo}) are each phg conormal on $X_b^{3}(s,x,w)\times N_{y}
\times S^{n-1}_{\zeta/|\zeta|}$. Moreover, the index sets are independent of $j$,
as $s/w>1$, and there is always infinite-order decay at $s=\infty$, independent
of $w$. As before, we use the pushforward theorem to integrate in $s$, 
and we conclude that the coefficients (\ref{conormcoefftwo}) are phg conormal on
$X_{b}^{2}(x,w)\times N_{y}\times S^{n-1}_{\zeta/|\zeta|}$ with respect
to index sets independent of $j$. Since $w>x/2$, this yields expansions
in $\frac{x}{w}$ and $w$, which is precisely what we need.

On the other hand, suppose that $w<2x$. Then we expect a conormal singularity
at $\hat X=\hat Y=0$. Since $(X,Y)=(s/x)(\hat X,\hat Y)$, we change variables
in (\ref{conormformtwo}), letting $\zeta'=(s/x)\xi$. Following the exact
same procedure as in the $w>x/2$ case, we see that the part of (\ref{gammatwo})
coming from $R_{2}$, when $w<2x$, is:
\begin{equation}\label{conormintthree}
\int_{\mathbb R^{n}}e^{i(\hat X,\hat Y)\cdot(\zeta'_{1},\zeta'_{2})}
\sum_{j=0}^{\infty}(\int_{w}^{\infty}2se^{-(s/w)^{2}e^{i\varphi}}
b_{j}(\frac{x}{s},s,y,\frac{\zeta'}{|\zeta'|})(s/x)^{j-2-n}\ ds|\zeta'|^{2-j}) 
d\zeta'.
\end{equation}
The coefficients can be rewritten as
\begin{equation}\label{conormcoeffthree}(w/x)^{j-2-n}
\int_{w}^{\infty}2se^{-(s/w)^{2}e^{i\varphi}}
b_{j}(\frac{x}{s},s,y,\frac{\zeta}{|\zeta|})(s/w)^{j-2-n}\ ds.
\end{equation}
But these are just $(w/x)^{j-2-n}$ times (\ref{conormcoefftwo}).
$(w/x)^{j-2-n}$ is phg conormal on $X_{b}^{2}(w,x)$, so
(\ref{conormcoeffthree}) is also phg conormal on $X_{b}^{2}(w,x)$ for 
each $j$. Since we are only considering $w<2x$,
the orders only improve as $j$ increases. In particular, all the
coefficients are phg conormal on $X_{b}^{2}(w,x)$ with respect to subsets of
the index set of the $j=0$ coefficient. This is again sufficient to prove
that (\ref{conormintthree}) is phg conormal on $M^2_{w,sc}$ with a conormal singularity at the diagonal.

\medskip

Finally, consider the interior term $R_3$; it is the simplest of the lot, since $z$ and $z'$ are in a compact subset of $M$. We let $\eta$ be a dual variable to $z-z'$ and write
\begin{equation}\label{finalthingthree}
R_{3}\sim\int_{\mathbb R^{n}}e^{i(z-z')\cdot(\eta)}
\sum_{j=0}^{\infty}c_{j}(s,z,\frac{\eta}{|\eta|})|\eta|^{2-j}\ d\eta.
\end{equation}
Here the $c_j$ are phg conormal at $s=0$ and $s=\infty$, with index sets independent of $j$ at $s=0$ and $s=\infty$. Following the same procedure as in the previous two cases, simplified since $z-z'$ and $\eta$ are independent of $s$, we conclude that the part of (\ref{gammatwo}) coming from $R_3$ is 
\begin{equation}\label{conormintfour}
\int_{\mathbb R^n}e^{i(z-z')\cdot\eta}
\sum_{j=0}^{\infty}(\int_{w}^{\infty}2se^{-(s/w)^{2}e^{i\varphi}}
c_{j}(s,z,\frac{\eta}{|\eta|})\ ds|\eta|^{2-j})\ d\eta.
\end{equation}
Analyzing the coefficients, we see that the integrand in each is phg conormal on $X_b^2(s,w)$, with index sets independent of $j$, $z$, and $\eta/|\eta|$, and with infinite-order decay at $s=\infty$. By the pushforward theorem, each coefficient is phg conormal at $w=0$, with index sets independent of $j$, $z$, and $\eta/|\eta|$. Moreover, (\ref{conormintfour}) has compact support in $(z,z')$. Therefore, (\ref{conormintfour}), and hence the part of (\ref{gammatwo}) corresponding to $R_c$, is phg conormal on $M^{2}_{w,sc}$ with a
conormal singularity at the diagonal.

Technically, we need to compute the index sets of the coefficients of the
conormal singularity at each
boundary face of $M^{2}_{w,sc}$. This may be done directly via the pushforward theorem, but it is easier to apply the analysis we will develop in the next section. Note that $a_{j}$ and $b_{j}$
may be viewed as phg conormal functions on the diagonal $D_{s,sc}\subset
M^{2}_{s,sc}$, with fixed index sets $R_{sc}$, $R_{bf_0}$, and $R_{zf}$
at the boundary hypersurfaces sc, bf$_0$, and zf. Observe that they can be extended
smoothly to functions defined in a neighborhood of $D_{s,sc}$ which are themselves
phg conormal on $M^{2}_{s,sc}$ with the given index sets;
call these extensions $d_{j}$. Then the coefficients
of the conormal singularity of the integral (\ref{gammatwo}) are the
restrictions to the diagonal of
\[
\int_{w}^{\infty}2se^{-(s/w)^{2}e^{i\varphi}} d_{j}(s,x,y,x',y')\ ds.\]
Applying Lemma \ref{biglemma} to each $d_j$, these coefficients are all phg conormal
on $M^{2}_{w,sc}$, with index sets obtained by adding $2$ at the faces bf$_0$, rb$_0$, lb$_0$, and zf. Therefore the restrictions to the diagonal are all phg conormal on $D_{w,sc}$, with
leading orders matching those in Theorem \ref{mainthm}, as expected. This completes the analysis of the conormal singularity.

\subsection{Finishing the proof}

It remains to consider the integral
\begin{equation}\label{gammatwosmooth}
\int_{w}^{\infty}2se^{-(s/w)^{2}e^{i\varphi}}
R_{s}(s^2,z,z')\ ds,\end{equation}
where $R_{s}(s^2,z,z')$ is phg conormal on $\bar M^{2}_{s,sc}$ and smooth
across the diagonal. We claim:

\begin{lemma}\label{biglemma} Let $T(s,z,z')$ be any function which is
phg conormal on $\bar M^{2}_{s,sc}$, smooth in the interior, and decaying to infinite
order at lb, rb, and bf. Then
\begin{equation}\label{genint}
\int_{w}^{\infty}2se^{-(s/w)^{2}e^{i\varphi}}T(s,z,z')\ ds
\end{equation}
is phg conormal on $M^{2}_{w,sc}$ for $w$ bounded above. Moreover, if the index sets of $T$ at the various boundary hypersurfaces are $T_{sc}$, $T_{bf_{0}}$, $T_{zf}$, $T_{rb_{0}}$, $T_{lb_{0}}$, and $T_{tf}$,
then the index sets of (\ref{genint}) are:
\[T_{sc},\ T_{bf_{0}}+2,\ T_{zf}+2,\ T_{rb_{0}}+2,\ T_{lb_{0}}+2.\]
\end{lemma}

We defer the proof for the moment. Applying Lemma \ref{biglemma} to $T(s)=R_{s}(s^2)$, we
conclude that (\ref{gammatwosmooth}) is phg conormal on $M^{2}_{w,sc}$ with index sets
precisely as in Theorem \ref{mainthm}.
Combining this with our analysis of  $R_c$, we 
have now shown that $F(w,z,z')$ is phg conormal on $M^{2}_{w,sc}$
possibly with a conormal singularity at the spatial diagonal, and
with leading orders as specified in Theorem \ref{mainthm}. However, $F(w,z,z')$ is a heat kernel, so it has no conormal singularity at the diagonal. This completes
the proof of Theorem \ref{mainthm}.

Finally, to prove Theorem \ref{maincor}, we apply Theorem \ref{mainthm}. Condition a) is true since the Laplacian is essentially self-adjoint and non-negative. Condition b) follows from Theorems \ref{ghlow} and \ref{dslow}. Condition c) is a well-known consequence of the semiclassical scattering calculus. The scattering calculus was first introduced by Melrose in \cite{me3} and the semiclassical version was developed by Vasy, Wunsch, and Zworski among others \cite{vz,wz}. The exact statement we need, along with a summary of the semiclassical scattering calculus, may be found in section 10 of \cite{hw}; $\hbar$ in the semiclassical calculus corresponds to $k^{-1}$ in our context. Applying Theorem \ref{mainthm} gives us the polyhomogeneity we claim, and once we plug in the leading orders from \cite{gh1} and the Appendix, we see that the heat kernel has leading orders of 0 at sc and $n$ at each of bf$_0$, rb$_0$, and lb$_0$.

Unfortunately, Theorem \ref{mainthm} does not by itself give us the claimed order-$n$ behavior at zf; instead, we only see at least quadratic decay at zf when $n\geq 3$ and decay of at least the form $w^2\log w$ when $n=2$. However, this may be improved via an a priori argument due to Pierre Albin \cite{albpc}. Specifically, since we know the heat kernel is polyhomogeneous, it has an expansion at zf of the form
\[F(w,z,z')=w^{s_0}(\log w)^j a_0(z,z')+\textrm{ (lower order terms)}\]
for some $s_0\in\mathbb R$, $j\in\mathbb N_0$, with $a_0(z,z')$ not identically zero. Applying the heat operator to this heat kernel gives zero by definition, and in these coordinates, the heat operator is $-w^2\partial_w+\Delta_z$. Since the kernel has a polyhomogeneous conormal expansion we may apply this operator term by term. The leading-order term of the result is 
\[w^{s_0}(\log w)^j\Delta_z a_0(z,z'),\]
with all other terms lower order. This implies that $\Delta_z a_0(z,z')$ must be zero, and therefore that $a_0(z,z')$ is harmonic for each $z'$, nonvanishing for at least some open set of values of $z'$. By the maximum principle, $a_0(z,z')$ cannot decay at infinity. So the leading order term of $F(w,z,z')$ at zf is $w^{s_0}(\log w)^j$ times a term which \emph{does not vanish} at the left face lf$_0$. Since $w^{s_0}(\log w)^j$ has index set $(s_0,j)$ at lf$_0$, the index set of $F(w,z,z')$ at lf$_0$ must contain a term no better than $(s_0,j)$ -- in particular $F(w,z,z')$ cannot decay to any order better than $w^{s_0}(\log w)^j$ at lf$_0$. However, we already know that the leading order of $F(w,z,z')$ at lf$_0$ is $n$. Thus $s_0\geq n$, with $j=0$ if $s_0=n$. This shows that the leading order of the heat kernel at zf must actually be at least $n$, completing the proof of Theorem \ref{maincor}.

Note that the lack of sharpness in the order calculation of Theorem \ref{mainthm} reflects the fact that our real-analytic approach does not take into account the complex-analytic structure of the resolvent; there is cancellation between the top and bottom parts of the integral that our approach cannot see. In particular, we could instead move the contour $\Gamma$ towards the spectrum and represent the heat kernel as an integral with respect to the spectral measure. Guillarmou, Hassell, and Sikora demonstrate in a recent paper that there is cancellation between the top and bottom parts of the contour in the spectral measure \cite{ghs}. In particular, the spectral measure at zf vanishes to order $n-1$ (\cite{ghs}, Theorem 1.2); integrating $e^{-\lambda t}$ against this spectral measure, we obtain an alternative proof of the fact that the heat kernel vanishes to order $n$ at zf. 

\subsection{Proof of Lemma \ref{biglemma}}

We now prove Lemma \ref{biglemma};
the proof involves extensive use of Melrose's pullback and pushforward theorems.
First, write $T(s,z,z')$ as $T_{1}+T_{2}$, 
where $T_{1}$ is supported away from sc and $T_{2}$ is supported 
in a neighborhood of sc. This partition is illustrated
in Figure \ref{aseven}. Then decompose (\ref{genint}) 
into two integrals, corresponding to $T_{1}$ and $T_{2}$.

\begin{figure}
\centering
\includegraphics[scale=0.50]{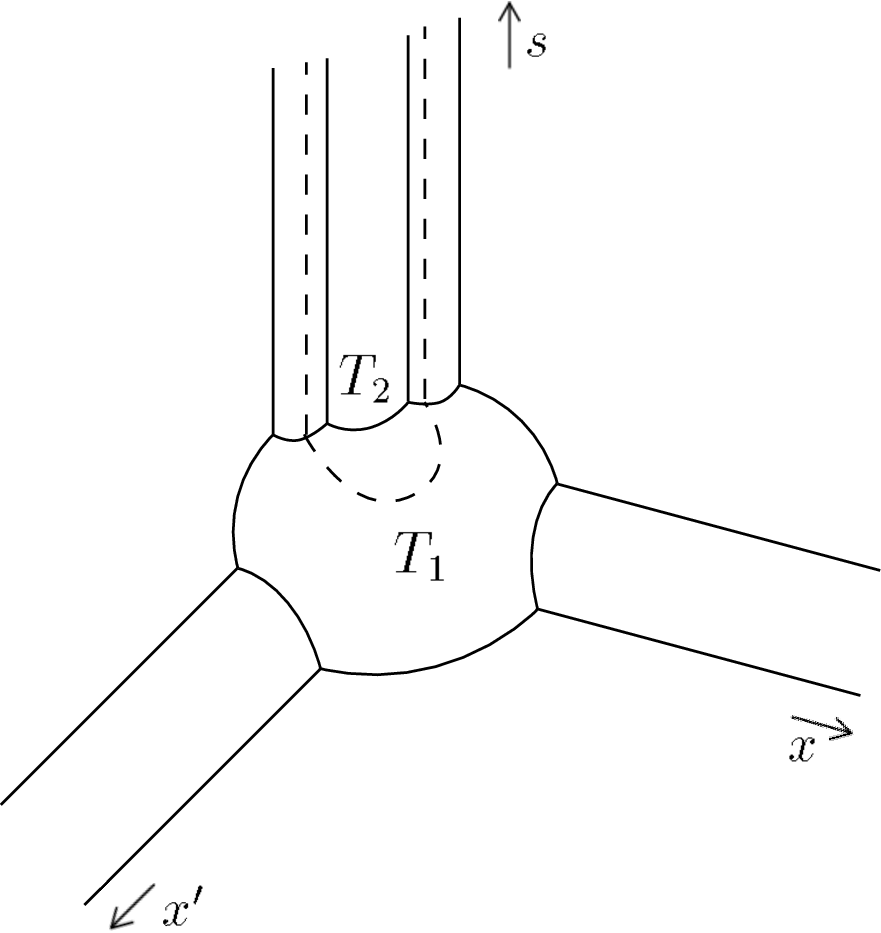}
\caption{Decomposition of $T$.}\label{aseven}
\end{figure}

Consider the first integral:
\begin{equation}\label{gammatwopartone}
\int_{0}^{\infty}\chi(\{s\geq w\})
2se^{-(s/w)^{2}e^{i\varphi}}T_{1}(s,z,z')\frac{ds}{s}.
\end{equation}
Notice that $T_{1}(s,z,z')$ is phg conormal on $M^2_{s,sc}$ but supported away from sc,
so it is in fact phg conormal on the blown-down space $M_{s,b}^{2}(z,z')$ (see Figure \ref{firstblowupgh}).
The rest of the terms in the integrand are phg conormal on $X_{b}^{2}(s,w)$, with a cutoff singularity (which is an example of a conormal singularity) at $s=w$. We now define a space $M_{s,w,b}^2(z,z')$ as follows: start with $[0,T)_{w}\times [0,\infty]_{s}\times M_{z}\times M_{z'}$. Then blow up, in order,
\begin{itemize}
\item The submanifold where all four of $(x,x',s,w)$ are zero;
\item The four now-disjoint submanifolds where exactly three of $(x,x',s,w)$ are zero;
\item The six now-disjoint submanifolds where exactly two of $(x,x',s,w)$ are zero.
\end{itemize}
This construction mimics the construction of the b-stretched product $X_b^4(x,x',s,w)$, and in fact $M_{s,w,b}^2(z,z')$ is precisely $X_b^4(x,x',s,w)\times N_y\times N_{y'}$ in a neighborhood of $\{x=x'=s=w=0\}$. By the same arguments as for the b-stretched products in \cite{ms}, the projection-induced maps from $M_{s,w,b}^2(z,z')$ to $M_{w,b}^2(z,z')$ (isomorphic to $X_b^3(x,x',w)\times N_y\times N_{y'}$ near $\{x=x'=w=0\}$) and to $X_b^2(s,w)$ are well-defined b-fibrations. Therefore, by the pullback theorem, the integrand of (\ref{gammatwopartone})
is phg conormal on $M_{s,w,b}^w(z,z')$, with
a conormal singularity at $s=w$. Since the fibers of the projection map to $M_{w,b}^2(z,z')$ are transverse to the singularity at $s=w$, and the integrand has order $\infty$ at $s=\infty$, the pushforward theorem implies that (\ref{gammatwopartone}) itself is phg conormal on $M^2_{w,b}(z,z')$. Since $M^2_{w,b}(z,z')$ is a blow-down of $M^2_{w,sc}(z,z')$, we conclude that (\ref{gammatwopartone}) is phg conormal on $M^2_{w,sc}(z,z')$ as desired.
\medskip

For $T_2$, we may use $(z,z')=(x,y,x',y')$ since $T_2$ is supported in a small neighborhood of sc. We have:
\begin{equation}\label{gammatwoparttwo}
(\int_{0}^{\infty}\chi(\{s\geq w\})
2se^{-(s/w)^{2}e^{i\varphi})}T_{2}(s,x,y,x',y')\frac{ds}{s}).
\end{equation}

Let $\bar{\sigma}=(\frac{x}{x'}-1,y-y')$; $\bar{\sigma}$ is an $n$-dimensional
coordinate, and $M^{2}_{s,sc}$ is created from $X_{b}^{3}(s,x,x')\times
N_{y}\times N_{y'}$ by blowing up $\{\bar\sigma=\frac{x}{s}=0\}$. In particular,
$T_{2}(s,x,y,x',y')$, having compact support in $x/x'$, is phg conormal on 
\[[X_{b}^{2}(s,x)\times N_{y};\{\bar\sigma=\frac{x}{s}=0\}].\]
This space is the subset of $M^{2}_{w,sc}$ with $1/2<x/x'<2$, so label
its boundary hypersurfaces bf, sc, bf$_{0}$, and zf. In this labeling,
$T_{2}$ is supported away from zf, decays to infinite order at bf,
and has leading orders $t_{sc}$ at sc and $t_{bf_{0}}$ at bf$_{0}$.

We analyze the integrand in (\ref{gammatwoparttwo})
as a function on the space
\[S=(X_{b}^{3}(s,w,x)\cap \{s\geq w\})\times B(\bar{\sigma})\times N_{y}.\]
Here $B$ is the unit ball in $\mathbb R^{n}$. A diagram of $S$ is given in
Figure \ref{thespaces},
with $\bar{\sigma}$ and $y$ suppressed; we label the boundary hypersurfaces
A-E.

\begin{figure}
\centering
\includegraphics[scale=0.60]{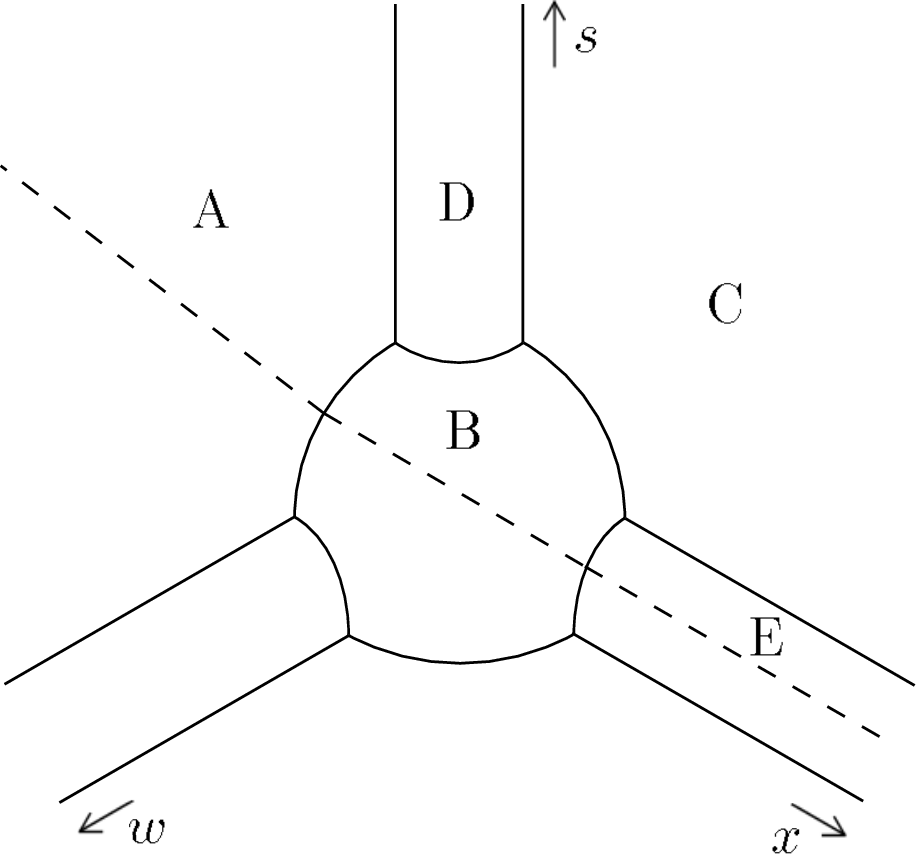}
\caption{The space $S$.}\label{thespaces}
\end{figure}

We now define an iterated blow-up of $S$.
Let $P_{1}$ be the p-submanifold of $S$
given by A$\cap \{\bar{\sigma}=0\}$. Blowing up $P_{1}$ creates a new space
$S_{1}=[S;P_{1}]$; call the front face of this blowup $F$.
Now let $P_{2}$ be the p-submanifold of $S_{1}$ given by the closure of the
lift of D$^{\circ}\cap\{\bar{\sigma}=0\}$. Then let 
\[S_{2}=[S_{1};P_{2}]=[[S;P_{1}];P_{2}],\]
and let $G$ be the new front face.
The following two propositions allow us to analyze
(\ref{gammatwoparttwo}); their proofs are deferred for the moment.
\begin{proposition}\label{propone} The map
\[\pi_{w}:S_{2}\cap\{s\geq x\}\rightarrow 
[(X_{b}^{2}(s,x)\cap \{s\geq x\})\times B_{1}(\bar{\sigma})
\times N_{y};\{\bar\sigma=\frac{x}{s}=0\}],\]
given in the interior of $S_{2}$ 
by projection off the variable $w$ and extending continuously to the
boundary, is a b-map.
\end{proposition}

\begin{proposition}\label{proptwo}  
The map
\[\pi_{s}:S_{1}\rightarrow 
[X_{b}^{2}(w,x)\times B_{1}(\bar{\sigma})\times 
N_{y};\{\bar\sigma=\frac{x}{w}=0\}],\]
given in the interior of $S_{1}$ 
by projection off the variable $s$ and extending continuously to the
boundary, is a b-fibration. Moreover, if we let $\rho_{H}$ be a bdf
for each hypersurface $H$, we have
\begin{equation}
(\pi_{s})^{*}(\rho_{zf})=\rho_{C}\rho_{E},
(\pi_{s})^{*}(\rho_{bf_{0}})=\rho_{B}\rho_{D},
(\pi_{s})^{*}(\rho_{sc})=\rho_{F},
(\pi_{s})^{*}(\rho_{bf})=\rho_{A}.\end{equation}
\end{proposition}

Since $T_{2}$ is supported in $\{s\geq x\}$ and its support does not intersect the lift of $\{s=x\}$, Proposition \ref{propone}
and the pullback theorem imply that the pullback of 
$T_{2}$ is phg conormal on $S_{2}$. Moreover, the remainder of the integrand
in (\ref{gammatwoparttwo}) is phg conormal on $(X_{b}^{2}(s,w)\cap\{s>w\})$, so
pulling back first to $X_{b}^{3}(s,w,x)\cap\{s>w\}$, and then to $S_{2}$, we see
that it is phg conormal on $S_{2}$ as well. Therefore, the entire integrand
in (\ref{gammatwoparttwo}) is phg conormal on $S_{2}=[S_{1};P_2]$. However, the factor of $e^{-s^2/w^2}$, and hence the integrand, vanishes to infinite order at the front face G; consequently the integrand in (\ref{gammatwoparttwo}) is actually phg conormal on $S_{1}$.
By the pushforward theorem from \cite{emm}, since $\pi_{s}$ is a b-fibration transverse to the
conormal singularity at $s=w$, the pushforward (\ref{gammatwoparttwo})
is phg conormal on the target space 
$[X_{b}^{2}(w,x)\times B_{1}(\bar{\sigma})\times 
N_{y};\{\bar\sigma=\frac{x}{w}=0\}]$. From Figure \ref{rangespace}, we see that this space
is a subset of $M^{2}_{w,sc}$; we have therefore shown that (\ref{gammatwoparttwo})
is phg conormal on $M^{2}_{w,sc}$. This completes the proof of the polyhomogeneity statement in Lemma \ref{biglemma}, modulo the proofs of Propositions \ref{propone} and \ref{proptwo}.

\begin{figure}
\centering
\includegraphics[scale=0.60]{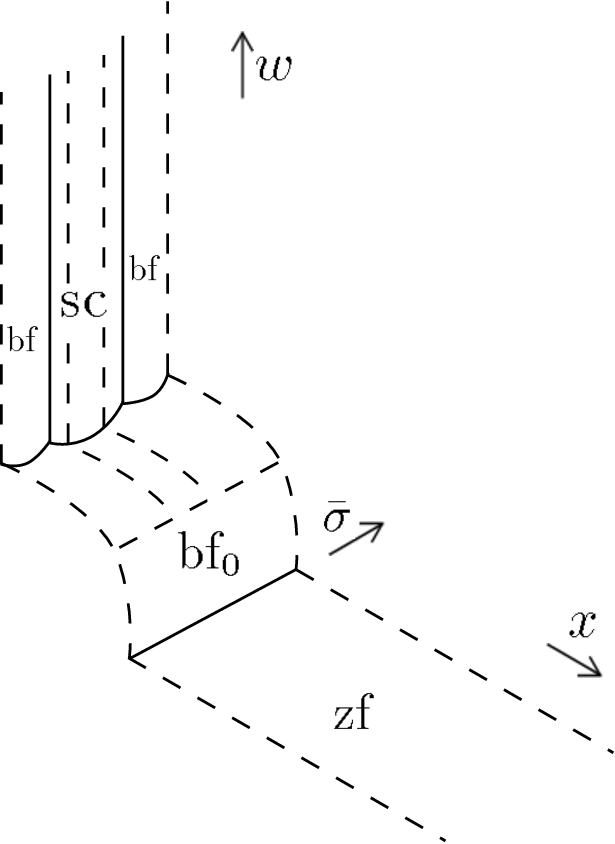}
\caption{$[X_{b}^{2}(w,x);\{\bar\sigma=\frac{x}{w}=0\}]$.}\label{rangespace}
\end{figure}

It remains to check the index sets claimed in Lemma \ref{biglemma}. However, this calculation is a straightforward application of the pullback and pushforward theorems (explicit descriptions of the pullback and pushforward index sets may be found in \cite{gri}). A computation of the leading orders may be found in \cite{s} and computing the index sets themselves is no harder.

\subsection{Propositions \ref{propone} and \ref{proptwo}}

Finally, we prove Propositions \ref{propone} and \ref{proptwo}. These propositions are proved in \cite{s} using explicit local coordinates, but here we instead give a simpler proof based on the machinery developed by Hassell-Mazzeo-Melrose in \cite{hmm}.

Observe first that there are projection-induced maps from $S$ to both $X_b^2(s,x)\times B(\bar\sigma)\times N_y$ and $X_b^2(w,x)\times B(\bar\sigma)\times N_y$. We call these maps $\tilde\pi_w$ and $\tilde\pi_s$ respectively. It is easy to see directly that both of these maps are in fact b-fibrations; see also the analysis of b-stretched products in \cite{ms}. Moreover, it may be checked by hand (also see \cite{ms}) that each entry of the 'exponent matrix' associated to each of these maps is either 0 or 1; see \cite{gri} or \cite{ma} for a discussion of exponent matrices.

To prove Proposition \ref{propone}, consider the p-submanifold $\{\bar\sigma=x/s=0\}$ of the target space of $\tilde\pi_w$. Its lift under $\tilde\pi_w$ is a union of two p-submanifolds of $S$: A$\ \cap\ \{\bar\sigma=0\}$ and D$\ \cap\ \{\bar\sigma=0\}$. $S_2$ is precisely the space we obtain from $S$ by blowing up those two p-submanifolds (first A, then D). We may therefore apply Lemma 10 in Section 2 of \cite{hmm} to conclude that the lift of $\tilde\pi_w$ to a map from $S_2$ to $[X_b^2(s,x)\times B(\bar\sigma)\times N_y; \{\bar\sigma=x/s=0\}]$ is a b-fibration; but this lift is precisely $\pi_w$. Since a b-fibration is certainly a b-map, this completes the proof of Proposition \ref{propone}.

Proposition \ref{proptwo} is proved in exactly the same way: the lift of $\{\bar\sigma=x/w=0\}$ to $S$ under $\tilde\pi_s$ is just A$\ \cap\ \{\bar\sigma=0\}$, which is precisely $P_1$. An identical application of Lemma 10 of \cite{hmm} allows us to conclude that $\pi_s$ is a b-fibration. The computation of the pullbacks of boundary defining functions is not hard and may be done directly using local coordinates; the details may be found in \cite{s}.

\section{Renormalized heat trace and zeta function}
In this section, we define the renormalized heat trace, zeta function, and determinant on an asymptotically conic manifold $M$. These definitions ultimately allow us, in \cite{s2}, to state and prove Theorem \ref{detapprox}. The first step is to define the \emph{renormalized trace}. 
This definition is inspired by Melrose's b-heat trace, which is a renormalized heat 
trace for manifolds with asymptotically cylindrical ends. 
Albin also defined renormalized heat traces in the
asymptotically hyperbolic setting \cite{alb}; later, Albin, Aldana, and Rochon
defined and investigated a renormalized determinant of the Laplacian
on asymptotically hyperbolic surfaces \cite{aar}.

\subsection{The renormalized heat trace}

Pick any cutoff function $\chi_{1}(r)$ on
$\mathbb R_{+}$ which is supported on $\{r\leq 2\}$ and equal to 
1 on $\{r\leq 1/2\}$. Assume that either

a) $\chi_{1}(r)$ is a non-increasing smooth function of $r$
(smooth cutoff), or

b) $\chi_{1}(r)$ is precisely the characteristic function of $[0,1]$
(sharp cutoff).

Then for any $\delta<1/2$,
let $\chi_{1,\delta}$ be a function on $M$, equal to $\chi_{1}(r\delta)$
for $r\geq 1$ and equal to 1 inside $\{r=1\}$. Consider the integral
\begin{equation}\label{partrenorm}
\int_{M}\chi_{1,\delta}(z)H^{M}(t,z,z)\ dz.\end{equation} 
We will prove the following theorem:
\begin{theorem}\label{totalrenorm} Let $\chi_1$ be either the smooth or the sharp cutoff.
The integral (\ref{partrenorm}) has a polyhomogeneous expansion in $\delta$ for each fixed $t$. Moreover, the finite part at $\delta=0$, which we denote $P(t)$, has polyhomogeneous expansions in $t$ at $t=0$ and $t^{-1}$ at $t=\infty$.
\end{theorem}

This theorem allows us to define the renormalized heat trace on an asymptotically conic manifold.
Roughly, this corresponds to integrating the heat kernel on the diagonal
over regions where $r\leq\delta^{-1}$, and renormalizing by subtracting the
divergent parts at $\delta=0$.
Renormalization in this fashion is often called Hadamard renormalization
(for details, see \cite{alb,aar}).

\begin{definition} Let $\chi_{1}(r)$ be the sharp cutoff.
The renormalized heat trace, denoted $^{R}Tr H^{M}(t)$, is 
the finite part at $\delta=0$ of (\ref{partrenorm}).\end{definition}

We now prove Theorem \ref{totalrenorm}. The key ingredient is the following observation on the structure of the heat kernel on the diagonal near the boundary, which is a consequence of the structure theorem we have proven for the heat kernel. The asymptotic structure of $H^{M}(t,x,y,x,y)$ reflected in this proposition is illustrated in Figure \ref{hkasymp}.

\begin{proposition}\label{thingtwo} 
a) For $t$ bounded above, 
$H^{M}(t,x,y,x,y)$ is phg conormal in $(\sqrt t,x)$, with smooth
dependence on $y$.

b) For $t$ bounded below,
let $w=t^{-1/2}$; then $H^{M}(t,x,y,x,y)$
is phg conormal as a function of $w$ and $x$ on $X_{b}^{2}(w,x)$, again with smooth dependence on $y$.
\end{proposition}

\begin{figure}
\centering
\includegraphics[scale=0.75]{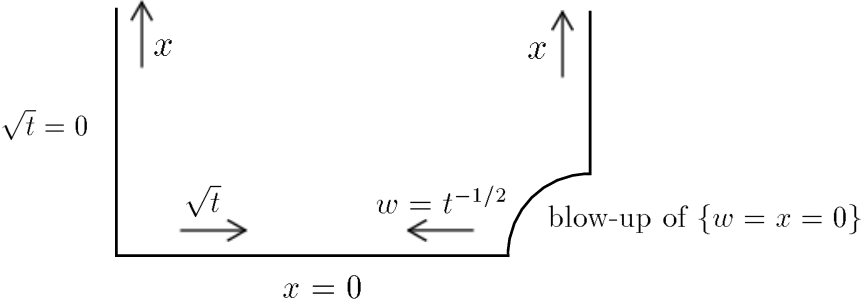}
\caption{Asymptotic structure of $H^{M}(t,x,y,x,y)$.}\label{hkasymp}
\end{figure}

\begin{proof} Proposition \ref{thingtwo} is an immediate consequence of restricting to the spatial diagonal $D$ in Theorem \ref{maincor}; since $D$ is a p-submanifold of the space on which $H^M$ is polyhomogeneous, the restriction of $H^M$ to $D$ is also polyhomogeneous. Comparing $D$ with Figure \ref{hkasymp}, we see that $D$ is precisely the space described in Proposition \ref{thingtwo}.
\end{proof}

To prove Theorem \ref{totalrenorm}, we analyze (\ref{partrenorm}), which may be rewritten as
\begin{equation}\label{thingthree}\int_{N}\int_0^1\chi_1(\delta/x)H^{M}(t,x,y,x,y)x^{-n-1}\ dx\ dy+\int_{x\geq 1}H^M(t,z,z)\ dz.\end{equation}
First analyze the second term; the region $\{x\geq 1\}$ is bounded away from spatial infinity. Therefore, by Theorem \ref{maincor}, $H^M(t,z,z)$ has polyhomogeneous expansions in $t$ at $t=0$ and $t^{-1/2}$, hence $t^{-1}$, at $t=\infty$, and these expansions are uniform in $z$ with smooth coefficients. Integrating in $z$ results in a function of $t$ which is phg conormal at $t=0$ and $t=\infty$; this function contributes only to the finite part $P(t)$ at $\delta=0$ and satisfies the polyhomogeneity claimed in Theorem \ref{totalrenorm}.

It remains to analyze the first term in (\ref{thingthree}).
We consider the small-$t$ and large-$t$ regimes separately, analyzing the
integrand \begin{equation}\label{integrand}\chi_1(\delta/x)H^M(t,x,y,x,y)x^{-n-1}\end{equation} in each regime as a function of $(t,x,\delta)$. In each case, $\chi_{1}(\delta/x)$ is phg
conormal on $X_{b}^{2}(x,\delta)$; if $\chi_1$ is the sharp cutoff, there is also a cutoff singularity,
which is a type of conormal singularity, at $\delta/x=1$.

For small $t$, $H^{M}(t,x,y,x,y)$ is phg conormal in $(\sqrt t,x)$, so (\ref{integrand}) is phg conormal on
$\mathbb R_{+}(\sqrt t)\times X_{b}^{2}(x,\delta)$, possibly with a conormal singularity at $\delta/x=1$. The projection
map $\pi_{x}$ is a b-fibration from this space onto 
the first quadrant in $(\sqrt t,\delta)$ and is transverse to $\delta/x=1$;
moreover, the integral in $x$ is well-defined, as the integrand is supported away from the $x=0,\delta>0$ face.
By the pushforward theorem from \cite{emm}, the first term of (\ref{thingthree}) is phg conormal in $(\sqrt t,\delta)$ for bounded $t$.

On the other hand, 
for large $t$, $H^{M}$ is phg conormal on $X_{b}^{2}(w,x)$ and $\chi_{1}$
is phg conormal on $X_{b}^{2}(x,\delta)$. Since the maps from $X_{b}^{3}
(w,x,\delta)$ to each of these spaces are b-maps (also b-fibrations), the
integrand is phg conormal on $X_{b}^{3}(w,x,\delta)$ by the pullback theorem; there may again be a conormal singularity at $\delta/x=1$. Integration in $x$ is pushforward by a b-fibration onto $X_{b}^{2}(w,\delta)$.
Again, the integrand is supported in $\{x>\delta\}$, and the fibration is transverse to $\delta/x=1$,
so we apply the pushforward theorem from \cite{emm} to conclude that the first term in (\ref{thingthree}) is phg conormal on $X_{b}^{2}(w,\delta)$
for bounded $w$.
Combining these results, we have shown that (\ref{thingthree}) is phg
conormal on the space in Figure \ref{hkasymptwo}.

\begin{figure}
\centering
\includegraphics[scale=0.75]{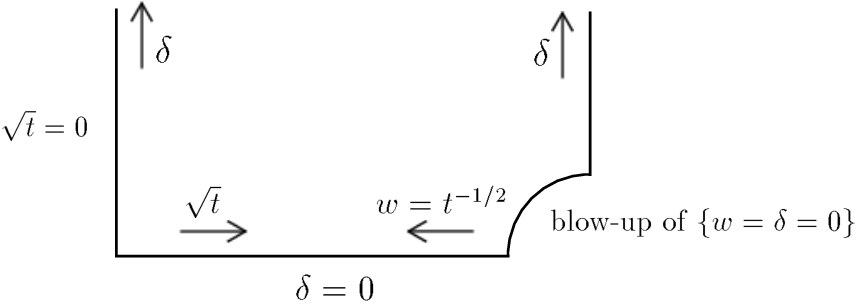}
\caption{Asymptotic structure of (\ref{thingthree}).}\label{hkasymptwo}
\end{figure}

In particular, for any fixed $t$, (\ref{thingthree}) has a polyhomogeneous expansion as $\delta\rightarrow 0$. Moreover, $P(t)$ is simply the coefficient of the $t^0$ term at the $(0<t<\infty,\delta=0)$ face. By the definition of phg conormality (also see the discussion surrounding Lemma A.4 in \cite{ma}), $P(t)$ therefore has polyhomogeneous conormal expansions at $t=0$ and $t=\infty$. This completes the proof of Theorem \ref{totalrenorm}.

\medskip

Using Theorem \ref{totalrenorm}, we now define the meromorphic continuation of the renormalized zeta function:
\begin{equation}\label{renzeta}
^{R}\zeta_{M}(s)=\frac{1}{\Gamma(s)}\int_{0}^{\infty}(^{R}Tr H^{M})(t)t^{s-1}\
dt.
\end{equation}

We break up the integral (\ref{renzeta})
at $t=1$, and consider first the short-time piece:
\begin{equation}\label{renzetashort}
\frac{1}{\Gamma(s)}\int_{0}^{1}(^{R}Tr H^{M}(t))t^{s-1}\ dt.\end{equation}
By the phg conormality of $^{R}Tr H^{M}(t)$, we can write for any $N>0$:
\[^{R}Tr H^{M}(t)=\sum_{i=0}^{k_{N}}a_{i}t^{z_{i}}(\log t)^{p_{i}}+
\mathcal O(t^{N}).\]
Plug this expansion into (\ref{renzetashort}).
The $\mathcal O(t^{N})$ contribution is well-defined and meromorphic whenever
$\Re s>-N$, and the continuations of the other terms are integrals of the
form
\[\frac{a_{i}}{\Gamma(s)}\int_{0}^{1}t^{z_{i}+s-1}(\log t)^{p_{i}}.\]
These integrals may be evaluated directly, and give explicit meromorphic
functions of $s$, each  with finitely many poles. Therefore,
(\ref{renzetashort}), though initially defined
only when $\Re s>-z_{0}$, has a meromorphic extension to
all of $\mathbb C$.

On the other hand, the long-time piece is
\begin{equation}\label{renzetalong}
\frac{1}{\Gamma(s)}\int_{1}^{\infty}(^{R}Tr H^{M}(t))t^{s-1}\ dt.\end{equation}
Writing $u=1/t$ and substituting, this becomes
\[\frac{1}{\Gamma(-(-s))}\int_{0}^{1}(^{R}Tr H^{M}(1/u))u^{(-s)-1}\ du.\]
We have a phg conormal expansion for $^{R}Tr H^{M}(1/u)$ as $u\rightarrow 0$;
say the leading order term is of the form $u^{z_{\infty}}(\log u)^{p}$.
Proceeding exactly as in the analysis of (\ref{renzetashort}), we conclude
that (\ref{renzetalong}), though initially defined only when 
$\Re(-s)>-z_{\infty}$, has a meromorphic continuation to all of $\mathbb C$.

This allows us to define the renormalized zeta function and determinant on any asymptotically conic manifold $M$.

\begin{definition} The renormalized zeta function on $M$, $^{R}\zeta_{M}(s)$,
is given by the meromorphic continuation of (\ref{renzeta}). \end{definition}
Depending on the orders, there may be no $s$ in $\mathbb C$
for which (\ref{renzeta}) is defined; however, once we split the integral
at $t=1$, both pieces are defined in half-planes and continue meromorphically
to all of $\mathbb C$.

\begin{definition} The \emph{renormalized
determinant} of the Laplacian on $M$ 
is $e^{-^{R}\zeta_{M}'(0)}$, where $^{R}\zeta_{M}'(0)$
is the coefficient of $s$ in the Laurent series for $^{R}\zeta_{M}(s)$ at
$s=0$. \end{definition}

\subsection{Manifolds conic near infinity}

We now specialize to the case of manifolds which are precisely conic outside a compact set. In particular, let $Z$ be any asymptotically conic manifold without boundary which is isometric to a cone outside a compact set. Without loss of generality, assume that $Z$ is isometric to a cone when $r\geq 1/2$. We examine the asymptotic expansion of 
$\int_{Z}\chi_{1,\delta}(z)H^{Z}(t,z,z)\ dz$ as $\delta\rightarrow 0$; the finite part is precisely the renormalized heat trace. However, for applications, such as the study of conic degeneration in \cite{s2}, we are also interested in identifying the divergent terms in the expansion. The fact that $Z$ is conic near infinity allows us to identify those terms:

\begin{theorem}\label{divexpansion} Let $Z$ be conic near infinity as above, and let $\chi_1$ be the sharp cutoff. Then
$\int_{Z}\chi_{1,\delta}(z)H^{Z}(t,z,z)\ dz$ has the following asymptotic expansion as $\delta\rightarrow 0$:
\begin{equation}\label{sharpexp}\int_{Z}\chi_{1,\delta}(z)H^{Z}(t,z,z)\ dz=
\sum_{k=0}^{n-1}f_{k}(t)\delta^{k-n}
+f_{log}(t)\log\delta+^{R}Tr H^{Z}(t)+R(\delta,t).\end{equation}
Here $R(\delta,t)$ goes to zero as $\delta$ goes to zero for each fixed $t$. Moreover, if we let $u_k(1,y)$ be the coefficient of $t^(k-n)/2$ in the short-time heat expansion on $C_N$ at the point $(1,y)$, then $f_{k}(t)=(\int_{N}u_{k}(1,y)\ dy)\frac{t^{(k-n)/2}}{k-n}$ and
$f_{\log}(t)=-\int_{N}u_{n}(1,y)\ dy$. \end{theorem}

We now prove Theorem \ref{divexpansion}. Note first that $\int_Z\chi_{1,\delta}(z)H^Z(t,z,z)$ does in fact have a polyhomogeneous expansion in $\delta$, by Theorem \ref{totalrenorm}, so it is just a matter of identifying the terms.
The proof involves a comparison of the heat kernels
on $Z$ and on $C_{N}$; $Z$ and $C_{N}$ are identical near infinity,
which allows us to formulate and prove the following lemma:

\begin{lemma} Let $C_{N}$ be the infinite cone over $N$. 
Then $|H^{C_{N}}(t,z,z)-H^{Z}(t,z,z)|$, defined whenever $r\geq 1$,
decays to infinite order in $|z|$ as $|z|$ goes to infinity.
\end{lemma}

\begin{proof} 
Let $\hat Z=Z\cap\{r\geq 1\}$. It is a
complete manifold with boundary at $r=1$, and is a subset of both $C_{N}$ and $Z$.
On $\hat Z$, $H^{C_{N}}(t,z,z')-H^{Z}(t,z,z')$ is a solution
of the heat equation for each $z'$, 
with initial data equal to zero and boundary data at $r=1$ given
by $H^{C_{N}}(t,1,y,z')-H^{Z}(t,1,y,z')$.
Fix any $T>0$. We claim that for all $t<T$, $y\in N$, 
and $z'$ with $|z'|>2$, there is a constant $K$ so that
the absolute value of the boundary data is less than $K$.

We show this for $H^{C_{N}}(t,1,y,z')$ and for $H^{Z}(t,1,y,z')$
separately. For $C_{N}$, by scaling and noting that $|z'|<2$,
\[H^{C_{N}}(t,1,y,z')=|z'|^{-n}H^{C_{N}}(t/|z'|^{2},1/|z'|,y,z'/|z'|)
<H^{C_{N}}(t/|z'|^{2},1/|z'|,y,1,y').\]
For each fixed $y'$, the heat kernel with point
source at $(1,y')$ is continuous for $r<1/2$ (i.e. the tip of the cone),
and hence is bounded for $r<1/2$ and for $t/|z'|^{2}<t<T$
by some universal constant $K$. Since $y'$ varies only over a compact set, 
the proof is complete.
As for $H^{Z}(t,1,y,z')$, consider the region 
\[W=\{(t,1,y,x',y')|\ t<T,x'<\frac{1}{2},y\in N,y'\in N\}\]
as a subset of $(t,x,y,x',y')$ space.
The kernel $H^{Z}$ has infinite-order decay at each boundary
hypersurface of the space in Figure \ref{hkac} with which $W$ has nontrivial
intersection. We conclude that $H^{Z}$ is bounded
on $W$, so there is a constant $K$ so that $H^{Z}(t,1,y,z')<K$ for all $t<T$, all $y\in N$, and all $z'$ with $|z'|>2$.

Since we have an upper bound for the boundary data,
we can construct a supersolution and apply the parabolic maximum principle.
Let $g(t,r)$ be the solution of the heat equation on $\hat Z$ with zero
initial condition and boundary data at $r=1$ equal to $K$ for all $t$.
By the maximum principle, we see that, uniformly for $|z'|>2$ and
$t<T$, $|H^{C_{N}}(t,z,z',t)-H^{Z}(t,z,z')|<g(t,|z'|)$.
We claim that $g(t,r)$ decays
to infinite order in $r$, uniformly in $t$ for $t<T$.
This can be seen either from Bessel function expansions
or by constructing a further
supersolution $\hat g(t,r)$
modeled on the heat kernel on $\mathbb R^{n}$. In particular, we can use
\[\hat g(t,r)=\frac{K/\alpha}{(4\pi)^{n/2}}
\sum_{k=-1}^{T_{0}}\frac{1}{(t-k)^{n/2}}e^{-\frac{r^{2}}{4(t-k)}}
\chi_{\{t>k\}},\]
where $\alpha=(8\pi)^{n/2}e^{1/4}$ and $T_{0}$ is the greatest integer less
than or equal to $T$. This supersolution has the uniform exponential decay
property we want, so a final application of 
the parabolic maximum principle finishes the proof of the lemma.
\end{proof}

The following corollary is an immediate consequence:

\begin{corollary} For any fixed $t$ and any $\chi_{1,\delta}$ (either
a sharp cutoff or a smooth cutoff),
\begin{equation}\label{difference}
|\int_{C_{N}} \chi_{1,\delta}(z)H^{C_{N}}(t,z,z)\ dz-
\int_{Z} \chi_{1,\delta}(z)H^{Z}(t,z,z)\ dz|\end{equation}
converges as $\delta\rightarrow 0$.
\end{corollary}
It now suffices to show that 
\[\int_{C_{N}}\chi_{1,\delta}(z)H^{C_{N}}(t,z,z)\ dz\]
has a divergent asymptotic expansion of the form claimed in Theorem
\ref{divexpansion},
as (\ref{difference}) converges as $\delta\rightarrow 0$ and hence
contributes only to the finite part of the expansion. (Recall that $\chi$ is the sharp cutoff).

\begin{lemma} Fix $t$. The divergent terms in the expansion of 
$\int_{|z|\leq 1/\delta}H^{C_{N}}(t,z,z)$
as $\delta\rightarrow 0$ are given by:
\[\sum_{k=0}^{n-1}(\int_{N}u_{k}(1,y)\ dy)\frac{t^{(k-n)/2}}{k-n}\delta^{k-n}
-C\log\delta,\] where $C$ is equal to $\int_{N}u_{n}(1,y)\ dy$.
\end{lemma}
\begin{proof} The integral is, modulo a term independent of $\delta$,
\[\int_{1}^{1/\delta}\int_{N}H^{C_{N}}(t,r,y,r,y)r^{n-1}\ dy\ dr.\]
By the conformal homogeneity of $C_{N}$,
$H^{C_{N}}(t,r,y,r,y)=r^{-n}H^{C_{N}}(t/r^{2},1,y,1,y)$. So the integral
becomes:
\[\int_{1}^{1/\delta}\int_{N}H^{C_{N}}
(\frac{t}{r^{2}},1,y,1,y)r^{-1}\ dy\ dr.\]
Now let $s=t/r^{2}$ and switch to an integral in $s$; we get:
\begin{equation}\label{coneint}\frac{1}{2}\int_{\delta^{2}t}^{t}(\int_{N}
H^{C^{N}}(s,1,y,1,y)s^{-1}\ dy)\ ds.\end{equation}
From short-time heat asymptotics, we know that
\begin{equation}\label{conestheat}(\int_{N}H^{C^{N}}(s,1,y,1,y)\ dy)=
\sum_{k=0}^{n-1}(\int_{N}u_{k}(1,y)\ dy)s^{(k-n)/2}+\int_Nu_n(1,y)\ dy+\mathcal O(s^{1/2}).
\end{equation}
Here $u_k(1,y)$ are the heat coefficients on the cone $C_N$ at the point $(1,y)$.
We plug (\ref{conestheat}) into (\ref{coneint}) and get
\[\sum_{k=0}^{n-1}(\int_{N}u_{k}(1,y)\ dy)
\frac{t^{(k-n)/2}}{k-n}\delta^{k-n}
-(\int_N u_n(1,y)\ dy)\log\delta+g(\delta,t),\] where $g(\delta,t)$ is finite
as $\delta\rightarrow 0$. This is what we wanted to prove. \end{proof}

Combining this lemma with the preceding corollary and the definition of the renormalized heat trace completes the proof of Theorem \ref{divexpansion}.

Finally, it is also useful to investigate the analogous divergent expansion when a smooth cutoff, rather than a sharp cutoff, is used.

\begin{lemma}\label{smoothexpa} Let $\chi_{1}(r)$ be as in a): smooth and non-increasing,
supported in $r\leq 2$ and $1$ when $r\leq 1/2$. Then:
\[
\int_{Z}\chi_{1,\delta}(z)H^{Z}(t,z,z)\ dz=
\sum_{k=0}^{n-1}l_{k}f_{k}(t)\delta^{k-n}+f_{log}(t)\log\delta\]
\begin{equation}\label{smoothexp}+(^{R}Tr H^{Z}(t)+l_{log}f_{log}(t))
+\tilde R(\delta,t),\end{equation}
where $l_{k}=-\int_{1/2}^{2}\chi_{1}'(r)r^{k-n}\ dr$,
$l_{log}=-\int_{1/2}^{2}\chi_{1}'(r)\log r\ dr$, and
$\tilde R(\delta,t)$ goes to zero as $\delta$ goes to zero for every
fixed $t$.
\end{lemma}

\begin{proof}
Let $\xi(r)$ be any function which is equal to a constant
$a$ for $r\leq 1/2$ and supported in
\{$r\leq 2\}$. For any $\delta<1/2$, we may define a function
$\xi_{\delta}(z)$ on $Z$ by letting $\xi_{\delta}(z)$ be equal to
$\xi{\delta r}$ for $r=|z|\geq 1$ and $a$ for $\{r\leq 1\}$.  
Then consider the integral
\begin{equation}\label{arbcutoff}
\int_{Z}\xi_{\delta}(z)H^{Z}(t,z,z)\ dz,\end{equation}
and examine its behavior as $\delta\rightarrow 0$.

When $\xi(r)$ is the characteristic function of $[0,1]$, we have
the expansion (\ref{sharpexp}). By replacing $\delta$ with $\delta/b$ for
any $b\in[1/2,2]$, we can compute the $\delta\rightarrow 0$
expansion of (\ref{arbcutoff}) for $\xi(r)$ equal
to the characteristic function of $[0,b]$. By linearity, we see that the
expansion of (\ref{arbcutoff}) for $\xi(r)=(\Delta h)\chi_{[a,b]}$ is:
\[\sum_{k=0}^{n-1}(\Delta h)f_{k}(t)(b^{k-n}-a^{k-n})\delta^{k-n}+
(\Delta h)f_{log}(t)
(\log b-\log a)\]\[+(\Delta h)(R(\delta/b,t)-R(\delta/a,t)).\]

Now let $\xi(r)=\chi_{1}(r)-\chi\{r\leq 1\}$; this is the difference
between the sharp and smooth cutoffs. Since the expansions of (\ref{arbcutoff})
are linear in $\xi$, we can approximate by step functions and then integrate
by summing over thin horizontal rectangles. Assume for simplicity
that $\chi_{1}(r)=1$ for all $r\leq 1$ (in the general case, there
are some negative signs, but we get the same answer). The thickness
of the rectangle at height $h$ is $\Delta h$. The length of the rectangle
is $\chi_{1}^{-1}(h)-1$. Putting all of this together, 
the expansion of (\ref{arbcutoff}) with respect to $\xi(r)$ is:
\[\sum_{k=0}^{n-1}f_{k}(t)(\int_{0}^{1}
((\chi_{1}^{-1}(h))^{k-n}-1)\ dh)\delta^{k-n}+f_{log}(t)(\int_{0}^{1}
(\log\chi_{1}^{-1}(h))\ dh+\tilde R(\delta,t),\]
where $\tilde R(\delta,t)$ is the contribution from the remainder terms.

Finally, perform the change of variables $u=\chi_{1}^{-1}(h)$,
then add the expansion for $\chi\{r\leq 1\}$; we obtain precisely
the expansion claimed in the statement of the lemma. This finishes
the proof, as long as
we can control the remainder term $\tilde R(\delta,t)$. Indeed, for each
fixed $t$, we claim that $\tilde R(\delta,t)$ goes to 
zero as $\delta$ goes to zero;
define a new function $S(\delta,t)$ by letting
\[S(\delta,t)=\sup_{1/2\leq\gamma\leq 2}|R(\delta/\gamma,t)|.\]
When $\xi(r)=(\Delta h)[a,b]$, the remainder is bounded
in absolute value by $(\Delta h)S(\delta,t)$.
So the integral from $h=0$ to $1$ is bounded
by $S(\delta,t)$, which goes to zero as $\delta\rightarrow 0$; this shows
boundedness of the remainder term and finishes the proof of the lemma.
\end{proof}

It is worth examining the dependence on the zeta function on the choice of cutoff $\chi_1$; we used a sharp cutoff to define it, but we could use a smooth cutoff instead. In this case, the finite part of the divergent $\delta$-expansion 
changes from $^{R}Tr H^{Z}(t)$
to $^{R}Tr H^{Z}(t)+l_{log}f_{log}(t)$. But $f_{log}$ and $l_{log}$ are
constants. So the renormalized heat trace only depends on the choice
of cutoff function by the addition of a constant, independent of $t$.
However, it can be easily shown by breaking up the integral at $t=1$ that
the meromorphic continuation of $\int_{0}^{\infty}Ct^{k}t^{s-1}\ dt$ 
is identically zero for any constant $C$. We have shown:

\begin{proposition} Let $Z$ be conic near infinity. The renormalized zeta function and determinant of the
Laplacian on $Z$ are independent of the choice of cutoff function 
$\chi_{1,\delta}$.
\end{proposition}

We have now shown the existence of a renormalized zeta function and determinant
of the Laplacian on any asymptotically conic manifold $M$; moreover, when $M$ is conic near infinity, we have computed the divergent terms in the expansion which leads to those renormalizations.

\section{The low-energy resolvent in two dimensions}
In this section, we extend the techniques used by Guillarmou and Hassell in \cite{gh1} to prove Theorem \ref{dslow}. In particular, we construct the low-energy resolvent on an asymptotically conic surface. The resolvent is
\[R(\theta,k,z,z')
=(\Delta_{M}+e^{i\theta}k^{2})^{-1}(z,z').\]
For simplicity, we set $\theta=0$, so that $R$ is a function of $(k,z,z')$. At the end of the section, we return to discuss allowing arbitrary $\theta\in[-\varphi,\varphi]$ and showing smoothness in $\theta$; however, this is not difficult.

\subsection{Strategy}

Our goal is to construct the Schwartz kernel of the resolvent, $R(k,z,z')$, as a
distribution on $M^{2}_{k,sc}$. To do this, as in \cite{gh1},
we will first construct a
parametrix $G(k)$ so that $(\Delta_{M}+k^{2})G(k)=Id+E(k)$, where $E(k)$ is an
error term. $G(k)$ will be a family (in $k$) of pseudodifferential operators
on $M$ whose Schwartz kernel is polyhomogeneous conormal on $M^{2}_{w,sc}$ with
an interior conormal singularity at the spatial diagonal.
By examining the leading order behavior of the equation
$(\Delta_{M}+k^{2})G(k)=Id$ at each boundary hypersurface of $M^{2}_{w,sc}$,
we obtain a model problem at each hypersurface. The leading order of the
parametrix $G(k)$ at each hypersurface should solve the model problem. We first
choose solutions of the model problem at each hypersurface, and then check
that they are consistent; that is, that they may be glued together to
obtain a parametrix $G(k)$. Finally, we analyze the error $E(k)$ and show
that it can be removed via a Neumann series argument.

In order to define the appropriate space of pseudodifferential operators,
we use certain density conventions, all the same as in \cite{gh1}. We consider 
$P=\Delta_{M}$ as an operator on \emph{scattering half-densities} by writing 
\[P(f(x,y)|x^{-n-1}dxdy|^{1/2})=(\Delta_{M}f)(x,y)|x^{-n-1}dxdy|^{1/2}.\]
As in \cite{msb} and \cite{gh1}, 
we expect a transition between scattering behavior for $k>0$ and b-behavior at $k=0$, 
which leads us to define the conformally related $b$-metric $g_{b}=x^{2}g$. 
The space $(M,g_{b})$ is asymptotically cylindrical. 
We then define $P_{b}=x^{-n/2-1}Px^{n/2-1}$ with
respect to scattering half-densities. However, we want to consider $P_{b}$
acting with respect to \emph{b-half densities} $|x^{-1}dx dy|^{1/2}$. After this
shift, the relationship between $P_{b}$ acting on $b$-half densities and 
$P$ acting on scattering half-densities is $P=xP_{b}x$.

Let $\tilde\Omega_{b}^{1/2}$ be the bundle of half-densities on $M^{2}_{k,sc}$
which is spanned by sections of the form
\[|f(k,x,y,x',y')\frac{dg_{b}dg_{b}'dk}{k}|^{1/2}.\]
Let $\nu$ be a smooth nonvanishing section of this bundle. Since it involves 
the b-metric $g_{b}$, $\tilde\Omega_{b}^{1/2}$ is not
the natural bundle near sc. In particular, the kernel of the identity operator
on $M$ has leading order $-n/2$ at sc with respect to $\tilde\Omega_{b}^{1/2}$
\cite{gh1}.
We can now define spaces of pseudodifferential operators, precisely as
in \cite{gh1}:

\begin{definition} Let $\rho_{sc}$ be a boundary defining function for sc.
The space $\Psi_{k}^{m,\mathcal E}(M;\tilde\Omega_{b}^{1/2})$
is the space of half-density kernels $K=K_{1}+K_{2}$ 
on $M^{2}_{k,sc}$ satisfying:

1) $\rho_{sc}^{n/2}K_{1}$ is supported near $\Delta_{k,sc}$,
and has an interior conormal singularity of order $m$
at $\Delta_{k,sc}$, with coefficients whose behavior at the boundary
is specified by $\mathcal E$;

2) $\rho_{sc}^{n/2}K_{2}$ is polyhomogeneous conormal on $M^{2}_{k,sc}$
with index family $\mathcal E$, and moreover decays to infinite order at bf,
lb, and rb.
\end{definition}
The factor of $\rho_{sc}^{n/2}$ corrects for the use of b-half densities near
sc. Using this definition, we can compute as in \cite{gh1} that
\[(P+k^{2})\in\Psi_{k}^{2,\mathcal E}(M;\tilde\Omega_{b}^{1/2}),\]
with index sets 0 at sc, 2 at bf$_{0}$, 0 at zf, 2 at lb$_{0}$, and 2 at rb$_{0}$.

As proven in \cite{gh1}, these spaces satisfy a composition rule:
\begin{proposition}\label{composition}
Suppose that $A\in\Psi_{k}^{m,\mathcal E}$ and
$B\in\Psi_{k}^{m',\mathcal F}$. Then $A\circ B$ is well
defined and an element of $\Psi_{k}^{m+m',\mathcal G}$, where
\[\mathcal G_{sc}=\mathcal E_{sc}+\mathcal F_{sc},
\mathcal G_{zf}=(\mathcal E_{zf}+\mathcal F_{zf})\bar{\cup}(\mathcal E_{rb_{0}}
+\mathcal F_{lb_{0}}),\mathcal G_{bf_{0}}=(\mathcal E_{bf_{0}}+\mathcal F_{bf_{0}})
\bar{\cup}(\mathcal E_{lb_{0}}+\mathcal F_{rb_{0}}),\]
\[\mathcal G_{lb_{0}}=(\mathcal E_{bf_{0}}+\mathcal F_{lb_{0}})
\bar{\cup}(\mathcal E_{lb_{0}}+\mathcal F_{zf}),
\mathcal G_{rb_{0}}=(\mathcal E_{rb_{0}}+\mathcal F_{bf_{0}})
\bar{\cup}(\mathcal E_{zf}+\mathcal F_{rb_{0}}).\]
\end{proposition}

We therefore expect our parametrix $G$ to be in $\Psi_{k}^{-2,\mathcal G}$
for some index family $\mathcal G$. To gain more information, we need
to start analyzing the model problems.

\subsection{The two-dimensional problem}

We begin our analysis of the model problems at the face zf. In order
to identify the leading-order part of the equation at a boundary hypersurface,
we need to pick a coordinate to use as a boundary defining function in the
interior of that face. For all the faces in the lift of $\{k=0\}$, we use
$k$, which is the easiest choice, since it commutes with $P$.
Since $P+k^{2}=xP_{b}x+k^{2}$, the leading order part of the operator
at zf, which we call the \emph{normal operator}, is $xP_{b}x$, and the
model problem is $(xP_{b}x)G^{0}_{zf}=Id$. We therefore expect that $G^{0}_{zf}$
will be $(xx')^{-1}$ times some right inverse for $P_{b}$.

In order to invert $P_{b}$, we use the b-calculus of Melrose \cite{me},
identifying zf, near bf$_0$, with the b-double space $X_{b}^{2}(x,x')\times N_{y}\times N_{y'}$.
The corner zf$\cap$bf$_{0}$ corresponds to the front face ff in the b-double space.
An easy calculation, following \cite{gh1}, shows that 
\[P_{b}=-(x\partial_{x})^{2}+(n/2-1)^{2}+\Delta_{N}+W,\] 
where $W$ is a lower-order term; that is, $W$ vanishes as a b-differential
operator at $x=0$. In fact, $P_{b}$ is an elliptic b-differential operator,
and hence may be inverted by 
following the procedure of Melrose, which is described
in \cite{me} and \cite{ma}. 

The first step in this procedure
is to consider the indicial operator, which is the
leading order part of $P_{b}$ at the front face ff. Using the coordinates
$(\sigma'=x/x',x',y,y')$, this is
\[I_{ff}(P_{b})=-(\sigma'\partial_{\sigma'})^{2}+(n/2-1)^{2}+\Delta_{N}.\]
With this terminology, the key theorem is as follows:

\begin{theorem}\label{relindex}\cite{me} 
$P_{b}$ is Fredholm as an operator between
$x^{\delta}H^{2}_{b}$ and $x^{\delta}L^{2}_{b}$ if and
only if $\delta$ is not an \emph{indicial root} of $I_{ff}(P_{b})$.
(See \cite{me} or \cite{ma} for definitions of $x^{\delta}H^{2}_{b}$
and $x^{\delta}L^{2}_{b}$).
\end{theorem}

In our setting, as in \cite{gh1}, the indicial roots are precisely
\[\pm\nu_{i}=\pm\sqrt{(n/2-1)^{2}+\lambda_{i}};
\lambda_{i}\in\sigma(\Delta_{N}).\]
When $n>2$, 0 is not an indicial root, and 
Guillarmou and Hassell show that $P_{b}$ is not only Fredholm
but invertible for $\delta=0$, and then set $G^{0}_{zf}$ to be
$(xx')^{-1}$ times that inverse.
However, in our case, $n=2$, so $\nu_{i}=\sqrt{\lambda_{i}}$,
and 0 is an indicial root. So $P_{b}$ is not even Fredholm
from $H^{2}_{b}$ to $L^{2}_{b}$. This is precisely
why the $n=2$ case is not considered in \cite{gh1}.

\subsection{An example: Euclidean space}

In order to gain some intuition for the behavior of the resolvent near zf
in the $n=2$ setting, we examine the simplest case, which is $M=\mathbb R^{2}$.
The resolvent on $\mathbb R^{2}$, acting on scattering half-densities
$|dz|=|x^{-3}dxdy|$, is
\[-\frac{1}{2\pi}H_{0}(k|z-z'|)=-\frac{1}{2\pi}H_{0}(k|\frac{y}{x}-\frac{y'}{x'}|),\]
where $H_{0}$ is the Hankel function of order zero.
From the asymptotics of the Hankel function, we know that $H_{0}(r)$
decays exponentially as $r\rightarrow\infty$, and for small $r$:
\[H_{0}(r)\sim -\log r+\log 2-\gamma+\mathcal{O}(r).\]

Using these asymptotics, one can show that the resolvent on $\mathbb R^{2}$
is phg conormal on $M^{2}_{k,sc}$. We are most interested in the leading order
behavior near zf. In a neighborhood of zf, we have $k|z-z'|<1$, so this leading
order behavior is controlled by the small-$r$ asymptotics
\[H_{0}(k|z-z'|)\sim -\log(k|z-z'|)+\log 2-\gamma=
-\log k-\log |z-z'|+\log 2-\gamma.\]
Some observations on these asymptotics:

- As we approach zf, the resolvent increases logarithmically. This is a
major difference from the $n\geq 3$ case studied in \cite{gh1}, in which
the resolvent is continuous down to zf. On the other hand, the resolvent is
continuous down to bf$_{0}$, lb$_{0}$, and rb$_{0}$.

- The function $-\frac{1}{2\pi}\log |z-z'|$ 
is the Green's function for the Laplacian
on $\mathbb R^{2}$.

These observations suggest that in two dimensions, we will have a logarithmic term at zf in addition to a zero-order term. We write these terms as $G_{zf}^{0,1}\log k$ and $G_{zf}^0$ respectively. From the Euclidean-space example, we expect to have that on scattering half-densities:
\[G_{zf}^{0,1}\log k+G_{zf}^0=-C\log k+F(z,z'),\]
where $F(z,z')$ is a right inverse for the operator $P$. Moreover, since $-C\log k$ has
logarithmic growth at bf$_{0}$, lb$_{0}$, and rb$_{0}$ but the resolvent
on $\mathbb R^{2}$ does not, we expect that $F(z,z')$ will
have logarithmic growth at those faces, with the right coefficient
to cancel the logarithmic growth coming from $-C\log k$.

\subsection{Construction of the initial parametrix}

We now construct our parametrix $G(k)$ by specifying its leading
order behavior at each boundary hypersurface and then checking that the models
are consistent.

\subsubsection{The diagonal, sc, and bf$_{0}$}

The resolvent has an interior conormal singularity at the diagonal 
$\{z,z'\}$. The symbol of $P+k^{2}$ is $|\eta|^{2}+k^{2}$,
where $\eta$ is the dual variable of $z-z'$. One can compute that
$|\eta|^{2}+k^{2}$ is elliptic
on $M^{2}_{k,sc}$, with leading orders $0$ at sc, $2$ at bf$_{0}$, and 0 at zf.
As in \cite{gh1}, we let the symbol of $G(k)$ be the inverse of $|\eta|^{2}+k^{2}$, in the
sense of operator composition. This determines the diagonal symbol of $G(k)$
up to symbols of order $-\infty$, and hence determines
$G(k)$ up to operators with smooth Schwartz kernels
in the interior of $M^{2}_{w,sc}$.

At sc, the analysis is identical to that in \cite{gh1}, so we omit some of
the details. The key point is that
sc can be described as a fiber bundle with $\mathbb R^{n}$ fibers, 
parametrized by $y'\in N$ and $k$. 
The normal operator of $P+k^{2}$ is $\Delta_{\mathbb R^{n}}+k^{2}$,
which has a well-defined inverse for $k>0$. In each fiber, we let
\[G_{sc}^{0} = (\Delta_{\mathbb R^{2}}+k^{2})^{-1}.\]

At bf$_{0}$, we again follow \cite{gh1} exactly.
We use the coordinates $(\kappa=k/x,\kappa'=k/x',y,y')$, 
with $k$ a bdf for bf$_{0}$. Note that these are only good coordinates 
on the interior of bf$_{0}$ - for example, they become degenerate
near zf. We then view the interior of bf$_{0}$ as $\mathbb R_{+}(\kappa)\times
\mathbb R_{+}(\kappa')\times N_{y}\times N_{y}'$. 
As in \cite{gh1}, the normal operator at bf$_{0}$ is
\[I_{bf_{0}}(k^{-2}(P+k^{2}))=\kappa^{-1}(-(\kappa\partial_{\kappa})^{2}
+\Delta_{N}+\kappa^{2})\kappa^{-1}.\]
Letting $P_{bf_{0}}=-(\kappa\partial_{\kappa})^{2}+\Delta_{N}+\kappa^{2}$, 
the model problem is $(\kappa P_{bf_{0}}\kappa) G^{-2}_{bf_{0}}=\delta{\kappa-\kappa'}
\delta{-y'}$. To solve it, we separate variables and invert $P_{bf_{0}}$.
For each eigenvalue $\lambda_{j}$ of $\Delta_{N}$, 
write $\nu_{j}=\sqrt\lambda_{j}$ (these are the indicial roots).
Let $E_{\nu_{j}}\subset L^{2}(N)$ be the corresponding eigenspace of $\Delta_{N}$,
and let $\Pi_{E_{\nu_{j}}}$ be projection in $L^{2}(N)$ onto $E_{\nu_{j}}$.
Then the inverse of $P_{bf_{0}}$ is
\[Q_{bf_{0}}=\sum_{j=0}^{\infty}\Pi_{E_{\nu_{j}}}(I_{\nu_{j}}
(\kappa)K_{\nu_{j}}(\kappa')\chi_{\{\kappa'>\kappa\}}+I_{\nu_{j}}
(\kappa')K_{\nu_{j}}(\kappa)\chi_{\{\kappa'<\kappa\}}).\]
The only difference between our setting and \cite{gh1} is that we have
$\nu_{0}=0$ as opposed to $\nu_{0}>0$. We then set 
\[G_{bf_{0}}^{-2}=(\kappa\kappa')Q_{bf_{0}}.\]

We need to check consistency between $G_{sc}^{0}$ and $G_{bf_{0}}^{-2}$;
that is, we need to show that they agree to leading order in a neighborhood
of sc$\cap$bf$_{0}$. This proof is the same as in \cite{gh1};
the model problems and 
formal expressions for $G_{sc}^{0}$ and $G_{bf_{0}}^{-2}$ are identical.
We do have $\nu_{0}=0$, and $K_{0}(r)$ has different small-$r$
asymptotics from $K_{\nu_{j}}(r)$ for $\nu_{j}>0$; this will be
reflected in the asymptotics of $G_{bf_{0}}^{-2}$ near zf.
However, since $\kappa$ and $\kappa'$ both approach 
infinity near sc, only the large-$r$ asymptotics are relevant for this
consistency check, and the large-$r$ asymptotics of $I_{\nu}(r)$ and $K_{\nu}(r)$ are no different when $\nu=0$.

Technically, we also need to check consistency between the diagonal symbol
and the models at bf$_{0}$ and sc. However, this is also the same as in
\cite{gh1}; the models at bf$_{0}$ and
sc themselves satisfy elliptic pseudodifferential equations
given by the leading order part of $(P+k^{2})G(k)=Id$ at those faces.
As a result, their symbols at the diagonal are determined up to symbols of
order $-\infty$, and agree up to order $-\infty$ with the inverse of
$|\xi|^{2}+k^{2}$.

\subsubsection{The leading order term at zf}

At zf, the model problem with respect to b-half-densities is
\[(xP_{b}x)(G_{zf}^{0,1}\log k + G_{zf}^{0})=Id,\]
which translates to:
\[(xP_bx)G_{zf}^0=Id;\ \ (xP_bx)G_{zf}^{0,1}=0.\]
Translating our observations in the $M=\mathbb R^{2}$ case
to b-half-densities, we expect
\[G_{zf}^{0,1}\log k + G_{zf}^{0}=(xx')^{-1}(-C\log k+F(z,z')),\]
where $F(z,z')$ is a right inverse for $P_{b}$. We need to pick
the correct right inverse; in particular, if we have one right inverse, we may add any function of $z'$ to obtain another right inverse. The correct choice should have
logarithmic singularities at all faces and should be
consistent with our choice of $G_{bf_{0}}^{-2}$.
To check consistency, we need to show that $k^{0}(G_{zf}^{0,1}\log k + G_{zf}^{0})$ and
$k^{-2}G_{bf_{0}}^{-2}$ agree to leading order at ff$=$bf$_{0}\cap$zf,
which is the same as checking 
if $(xx')(G_{zf}^{0,1}\log k + G_{zf}^{0})$ and $Q_{bf_{0}}=(\kappa\kappa')^{-1}G_{bf_{0}}^{-2}$ agree
to leading order there.

First examine the leading order part of $Q_{bf_{0}}$ at zf.
When $x<x'$, we use the coordinates $(s,\kappa,x',y,y')$,
and we have, where $V=Vol(N)$,
\begin{equation}\label{smallslimit}
Q_{bf_{0}}=V^{-1}I_{0}(\kappa \sigma')K_{0}(\kappa)+\sum_{j=1}^{\infty}\Pi_{E_{j}}I_{\nu_{j}}
(\kappa \sigma')K_{\nu_{j}}(\kappa).\end{equation}
The boundary defining function for zf is $\kappa$, so we need
to examine the small-$\kappa$ asymptotics. For $\nu>0$ we know
by standard asymptotics of Bessel functions in \cite{wa} that
\[I_{0}(r)\sim 1; K_{0}(r)\sim -\ln r+\ln 2-\gamma;\]
\[I_{\nu}(r)\sim\frac{1}{\Gamma(\nu+1)}(\frac{r}{2})^{\nu}; 
K_{\nu}(r)\sim\frac{\Gamma(\nu)}{2}(\frac{r}{2})^{-\nu}.\]
Here $\gamma$ is the Euler-Mascheroni constant.
Plugging these asymptotics into (\ref{smallslimit}) shows that the
leading order term in $\kappa$ is:
\[V^{-1}(-\log\kappa+\log 2-\gamma)
+\sum_{j=1}^{\infty}\Pi_{E_{\nu_{j}}}\frac{(\sigma')^{\nu_{j}}}{2\nu_{j}}.\]

On the other hand, when $x>x'$, we use the coordinates $(\sigma,\kappa',x,y,y')$
and perform the same sort of calculations to obtain that the leading
order term in $\kappa$ is
\[V^{-1}(-\log\kappa'+\log 2-\gamma)
+\sum_{j=1}^{\infty}\Pi_{E_{\nu_{j}}}\frac{\sigma^{\nu_{j}}}{2\nu_{j}}.\]
The $x<x'$ and $x>x'$ cases may be combined; we see that the leading
order term of $Q_{bf_{0}}$ at zf is
\begin{equation}\label{frombfo}
V^{-1}(-\log k+\log 2-\gamma+\log x'+\chi_{\sigma'<1}\log \sigma')+\sum_{j=1}^{\infty}
\Pi_{E_{\nu_{j}}}\frac{e^{-\nu_{j}|\log \sigma'|}}{2\nu_{j}}.\end{equation}
We see immediately that we must have $C=V^{-1}$, and hence we set 
\[G_{zf}^{0,1}=-V^{-1}(xx')^{-1}.\] We then need to construct $G_{zf}^{0}$ so that $(xx')G_{zf}^{0}$ has leading order at bf$_{0}$ given by
\begin{equation}
V^{-1}(\log 2-\gamma+\log x'+\chi_{\sigma'<1}\log \sigma')+\sum_{j=1}^{\infty}
\Pi_{E_{\nu_{j}}}\frac{e^{-\nu_{j}|\log \sigma'|}}{2\nu_{j}}.\end{equation}

\medskip

To construct $G_{zf}^{0}$, we must first find
the correct right inverse
for $P_{b}$. Fix $\delta$ with 
$0<\delta<\nu_{1}$; then $P_{b}$
is Fredholm from $x^{-\delta}H^{2}_{b}$ to $x^{-\delta}L^{2}_{b}$
by Theorem \ref{relindex}. Following
the usual b-calculus construction in \cite{me} and \cite{ma}, 
we obtain a generalized inverse $Q_{b}^{-\delta}$. We claim:
\begin{lemma}
$P_{b}$ is surjective onto $x^{-\delta}L^{2}_{b}$.
\end{lemma}
The lemma implies that $Q_{b}^{-\delta}$ is an exact right inverse for $P_{b}$.

\begin{proof} By taking adjoints, the lemma is
equivalent to the statement that $P_{b}$ is
injective on $x^{\delta}L^{2}_{b}$. Suppose that $u|dg_{b}|^{1/2}$ is in
$x^{\delta}L^{2}_{b}$ and satisfies $P_{b}u=0$. By regularity of solutions
to b-elliptic equations, $u$ is phg conormal on $M$ near $x=0$; since
$u\in x^{\delta}L^{2}_{b}$, it decays to at least order $\delta$ at $x=0$.
On the other hand, since $P=xP_{b}x$ and $|dg_{b}|^{1/2}=x|dg|^{1/2}$,
$u|dg|^{1/2}$ is in the kernel of $P=\Delta_{M}$, and hence $\Delta_{M}u=0$.
By the maximum principle, $u=0$, which completes the proof of the lemma.
\end{proof}

The correct right inverse will be a slight modification of $Q_{b}^{-\delta}$.
In order to check consistency, we need to understand
the structure of $Q_{b}^{-\delta}$ near the front face ff=bf$_{0}\cap$zf.
This structure is described in detail in \cite{me} and \cite{ma}.
In particular, the leading order of $Q_{b}^{-\delta}$ at ff is precisely 
the indicial operator $I_{ff}(Q_{b}^{-\delta})$, which satisfies the equation
\begin{equation}\label{indicialeq}
-(\sigma'\partial_{\sigma'})^{2}+\Delta_{N})I_{ff}(Q_{b})=\delta(\sigma'=1,y=y').\end{equation}
Moreover, from \cite{me} and \cite{ma}, $I_{ff}(Q_{b}^{-\delta})$
has polyhomogeneous expansions at $\sigma'=0$ and $\sigma'=\infty$, with leading order
terms at worst $(\sigma')^{-\delta}$ at each end; that is, a small amount of
growth is allowed at $\sigma'=0$, and a small amount of decay is required at 
$\sigma'=\infty$.

We now separate variables and solve (\ref{indicialeq}) directly.
For each $j\geq 1$, $(\sigma')^{\pm\nu_{j}}$ span the kernel of
$-(\sigma'\partial_{\sigma'})^{2}+\nu_{j}^{2}$. Therefore,
the solutions corresponding to $E_{\nu_{j}}$
are combinations of $(\sigma')^{\nu_{j}}$ and $(\sigma')^{-\nu_{j}}$ away from $\sigma'=1$.
By the requirements at $\sigma'=0$ and $\sigma'=\infty$, our solution is a multiple of
$(\sigma')^{-\nu_{j}}$ for $\sigma'>1$ and of $(\sigma')^{\nu_{j}}$ for $\sigma'<1$. Using the matching
conditions at $\sigma'=1$ arising from the delta function singularity,
the solution on the eigenspace $E_{\nu_{j}}$ is
\[\Pi_{E_{\nu_{j}}}\frac{e^{-\nu_{j}|\log \sigma'|}}{2\nu_{j}}.\]

We have to consider $\nu_{0}=0$ separately;
the kernel of $-(\sigma'\partial_{\sigma'})^{2}$ is spanned by $1$ and $\log \sigma'$.
Because we require decay at $\sigma'=\infty$, the solution for $\sigma'>1$ must be zero. 
Then the matching conditions at $\sigma'=1$ 
imply that the solution is $\log \sigma'$ for $\sigma'<1$. Since projection
onto $E_{0}$ is simply $V^{-1}$, the zero-eigenspace
solution is $V\chi_{\{\sigma'<1\}}\log \sigma'$.
Therefore the leading order part of $Q_{b}^{-\delta}$ at ff is
\begin{equation}\label{indicialq}I_{ff}(Q_{b}^{-\delta})=
V^{-1}\chi_{\{\sigma'<1\}}\log \sigma' + \sum_{j=1}^{\infty} \Pi_{E_{\nu_{j}}}
\frac{e^{-\nu_{j}|\log \sigma'|}}{2\nu_{j}}.\end{equation}

Now compare (\ref{indicialq}) with (\ref{frombfo}). Let $\chi(z')$ be a smooth cutoff function on $M$, equal to 1 when $x'\in [0,1]$ and 0 whenever $x'\geq 2$. We see immediately that
if we let \[G_{zf}^{0}=(xx')^{-1}(Q_{b}^{-\delta}+V^{-1}\chi(z')\log x'+
V^{-1}(\log 2-\gamma)),\]
then $G_{zf}^{0,1}\log k+G_{zf}^{0}$ and $G_{bf_{0}}^{-2}$ are consistent. Additionally, $G_{zf}^0$ solves the model problem $(xP_{b}x)G_{zf}^{0}=Id$ at zf; the key is that any function of $z'$ is independent of $(x,y)$ and hence is in the kernel of $P_b$. Similarly, $G_{zf}^{0,1}\log k=-(Vxx')^{-1}\log k$ is in the kernel of $xP_bx$ and hence solves the model problem.. Moreover, the diagonal symbol is consistent with $G_{zf}^{0,1}\log k+G_{zf}^0$ for the same
reason that it is consistent with $G_{sc}^{0}$ and $G_{bf_{0}}^{-2}$.

\subsubsection{The model terms at rb$_{0}$}

Finally, we need to specify the leading-order behavior of the
parametrix at rb$_{0}$; in fact, we need to specify some lower-order
terms as well. We use the coordinates $(x,y,\kappa',y')$; the $\kappa'=0$ face
is rb$_{0}\cap$zf and the $x=0$ face is rb$_{0}\cap$bf$_{0}$.
There will be a term $G^{\nu_{j}-1}_{rb_{0}}$ for each $\nu_{j}$ in $[0,1)$. 
The model problem near this face, with $k$ as a boundary defining function,
is $(xP_{b}x)u=0$, so we need $P_{b}(xG^{-\nu_{j}}_{rb_{0}})=0$ for each 
$\nu_{j}\in [0,1)$.

First we focus on the model of order $-1$. We let
\[G^{-1}_{rb_{0}}=V^{-1}x^{-1}\kappa'K_{0}(\kappa'),\]
and claim that this is consistent with $G_{zf}^{0,1}\log k+G^{0}_{zf}$ and $G^{-2}_{bf_{0}}$.

To check consistency with $G_{zf}^{0,1}\log k+G^{0}_{zf}$, we need to show that the leading
order of $G_{zf}^{0,1}\log k+G^{0}_{zf}$ agrees with the leading order of $k^{-1}G^{-1}_{rb_{0}}$
at zf$\cap$rb$_{0}$. Recall that at rb$_{0}$, which corresponds to
$s=\infty$, $Q_{b}^{-\delta}$ decays to a positive order. So 
$(xx')^{-1}Q_{b}^{-\delta}$ has leading order greater than $-1$ at rb$_{0}$;
therefore, the leading order part of $G_{zf}^{0,1}\log k+G^{0}_{zf}$ at rb$_{0}$ is precisely
\[(xx')^{-1}(V^{-1}(-\log\kappa'+\log 2-\gamma)).\]
But by Bessel function asymptotics,
\[k^{-1}G^{-1}_{rb_{0}}=V^{-1}(xx')^{-1}K_{0}(\kappa')\sim
(xx')^{-1}V^{-1}(-\log\kappa'+\log 2-\gamma).\]
Therefore $G_{zf}^{0,1}\log k+G^{0}_{zf}$ and $G^{-1}_{rb_{0}}$ are consistent.

We must also check consistency of $G^{-1}_{rb_{0}}$ with $G^{-2}_{bf_{0}}$. 
Near rb$_{0}$,
\[k^{-2}G^{-2}_{bf_{0}}=V^{-1}(xx')^{-1}I_{0}(\kappa)K_{0}(\kappa')
+(xx')^{-1}\sum_{j=1}^{\infty}\Pi_{E_{\nu_{j}}}I_{\nu_{j}}
(\kappa)K_{\nu_{j}}(\kappa').\]
We are only interested in the order $-1$ part of this term. Since 
$x'$ and $\kappa$ both vanish to first order at rb$_{0}$, all the $j>0$
terms have leading order $-1+\nu_{j}$ at rb$_{0}$. 
Since $I_{0}(0)=1$, the order $-1$ part of $k^{-2}G^{-2}_{bf_{0}}$
at rb$_{0}$ is precisely
\[(xx')^{-1}V^{-1}K_{0}(\kappa')=k^{-1}(V^{-1}x^{-1}\kappa'K_{0}(x'))
=k^{-1}G^{-1}_{bf_{0}}.\]
We conclude that $G^{-1}_{rb_{0}}$ is consistent with the models at
zf and bf$_{0}$.

We also need to specify some lower order terms at $bf_{0}$; for this
we precisely follow Section 4 of \cite{gh1}. At zf, they need to match
with the asymptotics of $Q_{b}^{-\delta}$, and at bf$_{0}$, they need to match
with the higher order Bessel functions. Both of these involve only the nonzero
indicial roots, so the terms and arguments are identical to \cite{gh1}.
In particular,
for any $0<\nu_{j}<1$ in the indicial set, we let
\[G_{rb_{0}}^{\nu_{j}-1}=x^{-1}\frac{\kappa' 
K_{\nu_{j}}(\kappa')}{\Gamma(\nu_{j})2^{\nu_{j}-1}}
v_{\nu_{j}}(z,y'),\]
where $v_{\nu_{j}}(z,y')$ is in the kernel of $P_{b}$ with asymptotic
\[v_{\nu_{j}}(x,y,y')=(2\nu_{j})^{-1}\Pi_{E_{j}}x^{-\nu_{j}}+
\mathcal{O}(x^{-\nu_{j}-1}\log x).\]
The function $v_{\nu_{j}}$ 
is there to match with the asymptotics of $Q_{b}$ at $rb_{0}$,
as in Section 4 of \cite{gh1}. In fact, these models are consistent
with our models at bf$_{0}$ and at zf by precisely the same
argument as in \cite{gh1}; we will not repeat it here.

\subsection{The final parametrix and resolvent}

We have now constructed models at sc, bf$_{0}$, zf, and rb$_{0}$ which
are consistent with each other and also with the diagonal symbol.
Moreover, all the models decay to infinite order as we approach lb, rb, or bf.
Therefore, we specify our parametrix $G(k)$ to be any pseudodifferential operator
in $\Psi_{k}^{-2,\mathcal E}$ with kernel having 
the specified diagonal symbol and specified
leading-order terms at sc, bf$_{0}$, zf, and rb$_{0}$. The consistency we checked
guarantees that such an operator exists. The behavior of the kernel of $G(k)$
at lb$_{0}$ may be freely chosen as long as the leading-order term is order $-1$
and it matches with our models at zf and bf$_{0}$; a term of order $-1$ will,
however, be required.

Now let $E(k)=(P+k^{2})G(k)-Id$. Since $G(k)$ has diagonal symbol equal to
the inverse of the symbol of $P+k^{2}$, the Schwartz kernel of $E(k)$ is smooth
on the interior of $M^{2}_{w,sc}$. Moreover,
since $P+k^{2}$ is a differential operator,
the Schwartz kernel of $E(k)$ is phg conormal on $M^{2}_{w,sc}$.

- At lb, rb, and bf, the Schwartz kernel of $G(k)$ vanishes to infinite order
along with all derivatives, so the same is true of $E(k)$.

- At sc, $G^{0}_{sc}$ solves the model problem, 
so $E(k)$ has positive leading order at sc.

- At bf$_{0}$, $G(k)$ has order $-2$, but $G^{-2}_{bf_{0}}$ 
solves the model problem, and moreover $P+k^{2}$ vanishes to second order. 
Therefore $E(k)$ has positive leading order at bf$_{0}$.

- At zf, $G^{0}_{zf}$ and $G^{0,1}_{zf}$ solve the model problem, so $E(k)$ has positive
leading order.

- At lb$_{0}$, $G(k)$ has order $-1$. The variables 
$k$ and $x$ both vanish at lb$_{0}$, so $k^{2}G(k)$ has order $1$
and $xG(k)$ has order $0$. Since $P_{b}$ is a b-differential operator,
$P_{b}(xG(k))$ also has order 0, and hence $(P+k^{2})G(k)=(xP_{b}x+k^{2})G(k)$
has order 1. Since $Id$ is supported away from lb$_{0}$, $E(k)$ decays to
at least order 1 at lb$_{0}$.

- At rb$_{0}$, $G(k)$ has order $-1$, but all the terms $G^{\nu_{j}-1}$
for $\nu_{j}\in[0,1)$ solve the model problem. Therefore, the error
$E(k)$ has leading order at worst 0.

To summarize, if $\mathcal{E}$ is the index set for $E(k)$, we have shown:
\[\mathcal{E}_{sc}>0, \mathcal{E}_{zf}>0, \mathcal{E}_{bf_{0}}>0,\ 
\mathcal{E}_{lb_{0}}\geq 1, \mathcal{E}_{rb_{0}}\geq 0,\ 
\mathcal{E}_{lb}=\mathcal{E}_{rb}=\mathcal{E}_{bf}=\emptyset.\]

Now we iterate away the error. By Proposition \ref{composition},
$E(k)^{2}$ vanishes to positive order at all faces of $M^{2}_{k,sc}$;
suppose that the order of vanishing at each face is greater than $\epsilon>0$.
Again applying Proposition \ref{composition}, we see that for each $N\in\mathbb N$,
the order of vanishing of $E(k)^{2N}$ and $E(k)^{2N+1}$ at each face of $M^{2}_{k,sc}$
is greater than $N\delta$. Therefore the Neumann series
\[(Id +E(k))^{-1}=\sum_{i=0}^{\infty}E(k)^{i}\]
may be summed asymptotically, and the sum defines an element of 
$\Psi_{k}^{-\infty,\hat{\mathcal E}}$ for some index family $\hat{\mathcal E}$.

Finally, let $R(k)=G(k)(Id+E(k))^{-1}$; we see that $(P+k^{2})R(k)=Id$. Since
$P+k^{2}$ is invertible for all positive $k$, its only right inverse is the 
resolvent. We conclude that $R(k)$ is in fact the resolvent, and it is an
element of $\Psi_{k}^{-\infty,\mathcal R}(M,\tilde\Omega_{b}^{1/2})$ for
some index family $\mathcal R$.

In order to prove Theorem \ref{dslow}, we need to
perform this construction for any angle $\theta\in (-\pi,\pi)$,
not just for $\theta=0$. However, as Guillarmou and Hassell claim in \cite{gh1},
the construction is essentially unchanged. Indeed,
we just use $e^{i\theta/2}k$ as our boundary defining function for 
the $k=0$ faces instead of $k$, and correspondingly
change the model at $sc$ from $(\Delta_{\mathbb R^{2}}+k^{2})^{-1}$
to $(\Delta_{\mathbb R^{2}}+e^{i\theta}k^{2})^{-1}$.
The construction is then precisely analogous to the $\theta=0$ case;
by construction, the index sets are independent of $\theta$.
Moreover, by the continuity of the resolvent
outside the spectrum (also by construction), 
all the dependence on $\theta$ is smooth.
This completes the proof.

\subsection{Leading orders of the resolvent}

Since $R(k)=G(k)-G(k)E(k)+G(k)E(k)^{2}-\ldots$,
we can obtain some information about the leading orders of $R(k)$ at each
face. For $n\geq 3$, it is shown in \cite{gh1} that the leading orders
of $R(k)$ are the same as those of $G(k)$; we claim that the same is true
when $n=2$.

When $n=2$, $G(k)$ has leading orders $-1$ at lb$_{0}$ and rb$_{0}$,
order 0 at sc, order -2 at bf$_{0}$, and logarithmic growth at zf.
$E(k)$ has non-negative leading orders at all faces, and it is easy
to use Proposition \ref{composition} to show that 
the leading orders of $G(k)E(k)$ are no worse than those
of $G(k)$. Similarly, it may be shown that the leading orders of $G(k)E(k)^{l}$
are no worse than those of $G(k)$. Since $G(k)$ is fixed and
$Id-E(k)+E(k)^{2}-\ldots$ is asymptotically summable, the series
$G(k)-G(k)E(k)+\ldots$ is also asymptotically summable; therefore,
the leading orders of $R(k)$ are no worse than those of $G(k)$. The leading
order terms themselves may be affected by $G(k)E(k)$, but the orders are not.

To summarize, when $n=2$, the leading orders of the exact resolvent
$R(k)$ are no worse than 0 at sc, -2 at bf$_{0}$,
logarithmic at zf, and $-1$ at lb$_{0}$ and rb$_{0}$. 
When $n\geq 3$, the orders are at worst
0 at sc, -2 at bf$_{0}$, 0 at zf, and $\frac{n}{2}-2$
at lb$_{0}$ and rb$_{0}$, as in \cite{gh1}. However, these orders are for
the resolvent acting on b-half-densities, rather than the more natural
scattering half-densities. Switching to scattering half-densities
requires an order shift, adding $n/2$ at each of $\{x=0\}$ and $\{x'=0\}$.
So we need to add $n/2$ to the orders at lb$_{0}$ and rb$_{0}$, and $n$
at bf$_{0}$. The order at sc remains unchanged, because the extra
factor of $\rho_{sc}^{n/2}$ in the definition of the calculus
already incorporates the shift. So: viewing the resolvent as a scattering
half-density $|dgdg'|^{1/2}$ acting on scattering half-densities for each $k$,
or equivalently as a function acting on functions on $M$ by 
integration against $dg$, it has leading orders given by:
\begin{itemize}
\item $0$ at sc and $n-2$ at bf$_{0}$, rb$_{0}$, and lb$_{0}$:
\item $r_{zf}$ at zf, where $r_{zf}=0$ if $n\geq 3$ and
$r_{zf}=(0,1)$ (that is, leading order behavior of $\log\rho_{zf}$) if $n=2$.
\end{itemize}

\appendix

\section{Construction of the short-time heat kernel}

Albin has created a framework for the construction of the heat kernel on an asymptotically conic manifold; essentially all of the hard work involved in this construction has already been done in \cite{alb}. To complete the construction and prove Theorem \ref{zshorttime}, all we need to do is create an initial parametrix for the heat kernel. This construction is the content of this short appendix and is based on Section 5 of \cite{alb}, in which Albin constructs the heat kernel on an edge manifold.

The space in Theorem \ref{zshorttime}, which we call $S_{heat}$, is obtained by taking the manifold $M^2_{sc}\times[0,T)_t$ and then blowing up the $t=0$ diagonal. We call the scattering face of $M^2_{sc}\times[0,T)_t$ sf and call the front face at $t=0$ ff. The heat operator is precisely $\partial_t+\Delta_M$. Our goal is to create a parametrix which, to first order, solves the normal equations at sf and ff.

First analyze the situation at sf. As first discussed in \cite{me3} and elaborated upon in \cite{gh1}, sf has a Euclidean structure, parametrized by $y\in N$ and $t\in[0,T)$. By the same analysis as in \cite{gh1}, the normal operator at sf is precisely $\partial_t+\Delta_{\mathbb R^n}$. We then simply let the model at sf be the Euclidean heat kernel, $H^{\mathbb R^n}$. This is analogous to the construction in Section 5 of \cite{alb} for the edge setting, in which the model is the heat kernel on hyperbolic space times the heat kernel on the fiber.

The analysis at ff is also standard, since ff corresponds to the short-time regime on the interior of the manifold, where the heat kernel asymptotics are local. We know that $d(z,z')/\sqrt t$ is a good coordinate along ff, zero at the spatial diagonal and increasing to infinity as we approach the original $t=0$ face. We therefore let the leading-order model at ff be
\[H_{ff}=\frac{1}{(4\pi t)^{n/2}}e^{-\frac{(d(z,z'))^2}{4t}}.\] The choice of model at ff is again based on the Euclidean heat kernel, and is precisely the same as the choice of model in the edge setting \cite{alb}.

Each model vanishes to infinite order as we approach all boundary hypersurfaces other than sf and ff. Moreover, the models are consistent, as the leading orders of each are precisely the Euclidean heat kernel at sf $\cap$ ff. We may therefore pick a pseudodifferential operator whose Schwartz kernel agrees with our models to leading order at sf and ff, and decays to infinite order at all other boundary faces. In Section 4 of \cite{alb}, Albin proves a composition rule for time-dependent pseudodifferential operators whose kernels are polyhomogeneous conormal on $S_{heat}$ - our setting is the 'scattering' setting, which is included in his analysis. We then use this composition rule and an iteration argument, precisely as in Section 5 of \cite{alb}, to construct the heat kernel as a polyhomogeneous conormal distribution on $S_{heat}$. This completes the proof of Theorem \ref{zshorttime}.

\end{document}